\documentclass[12pt]{amsart}
\usepackage[letter paper, portrait, margin = 1in]{geometry}
\usepackage{psfrag}
\usepackage{color}
\usepackage{tikz}
\usetikzlibrary{matrix,arrows}
\usepackage{graphicx,graphics}
\usepackage{fullpage,amssymb,amsfonts,amsmath,amstext,amsthm,amscd,verbatim,enumerate}
\usepackage[T1]{fontenc}

\newtheorem{theorem}{Theorem}[section]

\newtheorem{lemma}[theorem]{Lemma}
\newtheorem{proposition}[theorem]{Proposition}
\newtheorem{corollary}[theorem]{Corollary}

\theoremstyle{definition}
\newtheorem{definition}[theorem]{Definition}

\newtheorem{example}[theorem]{Example}

\newcommand{\Z}{\mathbb{Z}}

\newcommand{\R}{\mathbb{R}}
\newcommand{\N}{\mathbb{N}}

\newcommand{\C}{\mathbb{C}}

\numberwithin{equation}{section}

\begin{document}

\title[]{A Complete Realization of the orbits of generalized derivatives of Quasiregular Mappings }
\date{\today}

\author{Alastair Fletcher}
\address{Department of Mathematical Sciences, Northern Illinois University, DeKalb, IL 60115-2888. USA}
\email{fletcher@math.niu.edu}
\thanks{The first named author was supported by a grant from the Simons Foundation, \#352034}

\author{Jacob Pratscher}
\email{jpratscher1@niu.edu}

%
%

\begin{abstract}
Quasiregular maps are differentiable almost everywhere maps which are analogous to holomorphic maps in the plane for higher real dimensions. Introduced by Gutlyanskii et al \cite{GMRV}, the infinitesimal space is a generalization of the notion of derivatives for quasiregular maps. Evaluation of all elements in the infinitesimal space at a particular point is called the orbit space. We prove that any compact connected subset of $\R^n\setminus\{0\}$ can be realized as an orbit space of a quasiconformal map. To that end, we construct analogues of logarithmic spiral maps and interpolation between radial stretch maps in higher dimensions. For the construction of such maps, we need to implement a new tool called the Zorich transform, which is a direct analogue of the logarithmic transform. The Zorich transform could have further applications in quasiregular dynamics. 
\end{abstract}
 
\maketitle
\theoremstyle{plain}

\section{Introduction}

Quasiregular mappings provide a natural setting for us to generalize the theory of holomorphic mappings in dimensions two into higher real dimensions. The generalized Liouville Theorem states that the only conformal mappings in $\R^n$, for $n\geq3$ are M\"obius mappings. Then to have a rich function theory, it is necessary to consider mappings with distortion. 

Fortunately, many of the important properties that holomorphic mappings and families of holomorphic mappings enjoy are also available for quasiregular mappings. Rickman generalized the little Picard Theorem to show that every non-constant entire quasiregular mapping can only omit finitely many values. Consequently, there is a direct analogue of Montel's Theorem which provides a criterion for a family of quasiregular mappings with uniformly bounded distortion to be normal. We refer to Iwaniec-Martin \cite{IM} and Rickman \cite{R} for introductions to the theory of quasiregular mappings. 

On the other hand, quasiregular mappings have more flexibility than holomorphic mappings. For example, there is no analogue of the Identity Theorem. Of particular relevance to this paper is that quasiregular mappings are only required to be differentiable almost everywhere. Our desire is to understand the behavior of quasiregular mappings near points where the derivative does not exist. 

To that end, Gutlyanskii et al \cite{GMRV} introduced the notion of generalized derivatives. These arise from a re-scaling procedure based on the facts that a quasiregular mapping is at worst H\"older continuous and, moreover, that there is only a bounded amount of distortion on all small spheres centered at the point of interest. The normal family machinery mentioned above is then used to conclude that generalized derivatives always exist. Generalized derivatives have been studied in \cite{FM, FMWW, FW, Miyachi}.

The main novelty compared to differentiable mappings is that there may be more than one generalized derivative at a given point. The collection of all generalized derivatives of $f$ at $x_0$ is called the infinitesimal space and denoted by $T(x_0,f)$. It was observed in Fletcher et al \cite{FMWW} that even in the well-behaved sub-class of uniformly quasiregular mappings, the infinitesimal space can contain uncountably many mappings. 

This property was shown by the first author and Wallis \cite{FW} to lead to a dichotomy: either the infinitesimal space contains one generalized derivative, or contains uncountably many. This was acheived by looking at the orbit of a point under all elements in the infinitesimal space and showing that if it contains more than one point, then it contains a continuum that must necessarily be contained in $\R^n\setminus\{0\}$. 

Conversely, it was also proved in \cite{FW} that in dimension two, any compact connected set in $\R^2\setminus\{0\}$ arises as an orbit space of some quasiconformal map. The methods used here strongly relied on computations involving the complex dilatation of a quasiconformal map, a tool only available in the plane. Our goal in this paper is to complete the realization of the orbit space by showing that every compact connected set in $\R^n\setminus\{0\}$, for $n\geq3$, arises as an orbit space of a quasiconformal map. Along the way we will need to construct new examples of quasiconformal mappings, in particular generalizing logarithmic spiral maps into higher dimensions. 

Since the complex dilatation is not available to us in higher dimensions, and direct calculations seem difficult to get a handle on, we will introduce a new technique for quasiregular mappings in higher dimensions called the Zorich transform. This is an analogue of the logarithmic transform in the plane. As an example of its utility, one of the maps we will need to construct involves interpolating in a round ring between boundary maps which are both radial stretches in the same direction, but of different factors. Applying the Zorich transform to this situation lead to a map which is the identity in $n-1$ components, yielding a much simpler situation. 

The logarithmic transform has been a highly useful tool in complex dynamics. It was introduced by Eremenko and Lyubich \cite{EL} and has found much utility in studying the class $\mathcal{B}$ of transcendental entire functions whose singular values are bounded, see also \cite[Section 5]{Sixsmith}. Another use of the logarithmic transform is in B\"ottcher's Theorem, see \cite{Milnor}, and its quasiregular generalization \cite{FF}. It should be noted that while we expect the Zorich transform to find further utility in studying quasiregular mappings, it is much more delicate than the logarithmic transform. For example, the associated map to the Zorich transform is the Zorich map, which in higher dimensions has non-trivial branching. As a consequence, the Zorich transform can typically only be defined locally, and not globally. 

The paper is organized as follows. In section two, we cover preliminary material on quasiregular mappings, generalized derivatives, and orbit spaces, where we finish the discussion of orbit spaces by stating our main result. In section three, we define and study properties of the Zorich Transform. In section four, we will construct the higher dimensional versions of the logarithmic spiral maps that we will later on want to link together. Finally, in section five we will put all of our results together and show that in dimension three and higher, we can realize any non-empty, compact, connected subset of $\R^n \setminus \{ 0 \}$ as an orbit space.

\section{Preliminaries}

\subsection{Quasiregular mappings}

Let us begin by  defining a quasiregular map. Note that details on linear distortion and distortion bounds can be found in Iwaniec and Martin, \cite{IM}. We will be using $\|\cdot\|$ to represent the operator norm, and $|\cdot|$ to represent the standard Euclidean norm. 

\begin{definition}
\label{def:qr}
Let $n \geq 2$ and $U$ a domain in $\R^n$. Then a continuous mapping $f:U \to \R^n$ is called quasiregular if $f$ is in the Sobolev space $W^1_{n, loc}(U)$ and there exists $K\in [1, \infty)$ so that 
\[  \|f'(x)\|^n \leq K J_f(x) \: \operatorname{a.e.}, \] where $f'$ is the derivative matrix,  and $J_f$ is the Jacobian of $f$. 
The smallest $K$ here is called the outer dilatation $K_O(f)$ of $f$. If $f$ is quasiregular, then it is also true that
\[ J_f(x) \leq K' \ell( f'(x) )^n \: \operatorname{a.e.}\]
for some $K' \in [1,\infty)$. Here, $\ell ( f'(x) ) = \inf_{|h|=1 } |f'(x)h|$. The smallest $K'$ for which this holds is called the inner dilatation $K_I(f)$ of $f$. The maximal dilatation is then $K(f) = \max \{ K_O(f) , K_I(f) \}$. We say that $f$ is $K$-quasiregular if $K(f) \leq K$.
\end{definition}

If $U$ is a domain in $\R^n$ with non-empty boundary, then for $x\in U$, we denote by $d(x,\partial U)$ the Euclidean distance from $x$ to $\partial U$.
One of the important properties of quasiregular mappings is that they have bounded linear distortion, which we now define.

\begin{definition}
\label{def:linear}
Let $n\geq 2$, $U \subset \R^n$ a domain, $x\in U$ and $f:U \to \R^n$ be $K$-quasiregular. For $0<r<d(x,\partial U)$, we define
\[ \ell_f(x,r) = \inf_{|y-x|=r} |f(y) - f(x) |, \: L_f(x,r) = \sup_{|y-x| = r} |f(y)-f(x)|.\]
The linear distortion of $f$ at $x$ is
\[ H(x,f) = \limsup_{r\to 0} \frac{L_f(x,r)}{\ell_f (x,r) }.\]
 When $f$ is differentiable at $x$ and the derivative matrix $f'$ is invertible, we can define the linear distortion, \cite[Section 6.4]{IM},  to be 
 \begin{equation}
 \label{defeq:lindistort}
 H(x,f)=\|f'\|\|(f')^{-1}\|. 
 \end{equation}
\end{definition}
We also know that the distortion of a $K$-quasiregular map $f$ is bounded by the linear distortion, that is 
\begin{equation}
\label{eq: kdistboundbyhlinear}
K\leq H^{n-1}.
\end{equation}  

 For our applications, we will estimate quantities related to the derivative of the maps we construct. However, our mappings will not be differentiable everywhere. To conclude quasiconformality, we will need to make use of the following theorem.  

\begin{theorem}
\label{thm: removability}
Suppose that $f:D\to D'$ is a homeomorphism, and that $E\subset D$ is a set such that $E$ is closed in $D$ and such that $E$ has a $\sigma$-finite $(n-1)$-dimensional Hausdorff measure. Suppose there exist constants $C_1,C_2>1$ such that $f$ is differentiable on $D\setminus E$ and
\begin{enumerate}
	\item either the matrix $f'$ has an inverse and $\sup_{x\in D\setminus E}H(x,f)\leq C_1$,
	\item or $\inf_{x\in D\setminus E}J_f(x)\geq\frac{1}{C_2}$  and $\sup_{x\in D\setminus E}\|f'(x)\|\leq C_2$. 
\end{enumerate}
Then we can conclude that $f$ is quasiconformal on $D$. 
\end{theorem}
\begin{proof}
	Given $x\in D\setminus E$, let $U\subset D\setminus E$ be a neighborhood of $x$.
	First suppose on $U$ that we satisfy condition $(ii)$ so that $J_f$ is bounded below by a positive number, and $\|f'\|$ is bounded on $U$. Note that Iwaniec and Martin,  \cite[Section 6.4]{IM}, tells us that 
		\[K_I\leq K_O^{n-1}. \] 
		Since $J_f$ is bounded below on $U$, we have that the matrix $f'$ is invertible. Since we also have $\|f'\|$ being bounded,  we can conclude that $K_O$ and hence $K(f)$ is bounded on $U$. That is, $f$ is quasiconformal on $U$. 
		Now suppose on $U$ we satisfy condition $(i)$ so that  the matrix $f'$ has an inverse and the linear distortion $H$ from \eqref{defeq:lindistort} is bounded by above on $U$. Then by \eqref{eq: kdistboundbyhlinear} the maximal dilatation 
		\begin{equation}
		K\leq(H')^{n-1}.
		\end{equation}
		Then $f$ is quasiconformal on $U$. So ever point $x\in D\setminus E$ has a neighborhood $U$ with maximal dilatation being bounded by a constant depending on $C_1$ or $C_2$, respectively. 
		 Then V\"ais\"al\"a, \cite[Theorem 35.1]{V}, tells us that  $f$ is quasiconformal on $D$. 
\end{proof}

The local index $i(x,f)$ of a quasiregular mapping $f$ at the point $x$ is
\[ i(x,f) = \inf_N \sup_{y\in N} \operatorname{card} ( f^{-1}(y) \cap N ),\]
where the infimum is taken over all neighborhoods $N$ of $x$. In particular, $f$ is locally injective  at $x$ if and only if $i(x,f) = 1$.

\begin{theorem}[Theorem II.4.3, \cite{R}]
\label{thm:lindist}
Let $n\geq 2$, $U\subset \R^n$ be a domain and $f:U \to \R^n$  be a non-constant quasiregular mapping. Then for all $x\in U$,
\[ H(x,f) \leq C < \infty,\]
where $C$ is a constant that depends only on $n$ and the product $i(x,f)K_O(f)$.
\end{theorem}

Recall that a family $\mathcal{F}$ of $K$-quasiregular mappings defined on a domain $U\subset \R^n$ is called normal if every sequence in $\mathcal{F}$ has a subsequence which converges uniformly on compact subsets of $U$ to a $K$-quasiregular mapping or to infinity.
There is a version of Montel's Theorem for quasiregular mappings due to Miniowitz.

\begin{theorem}[\cite{Mini}]
\label{thm:normal}
Let $\mathcal{F}$ be a family of $K$-quasiregular mappings defined on a domain $U \subset \R^n$. Then there exists a constant $q=q(n,K)$ so that if $a_1,\ldots, a_q$ are distinct points in $\R^n$ so that $f(U) \cap \{ a_1,\ldots, a_q \} = \emptyset$ for all $f\in \mathcal{F}$, then $\mathcal{F}$ is a normal family.
\end{theorem}

The constant $q$ here is called Rickman's constant and arises from Rickman's version of Picard's Theorem, see \cite[Theorem IV.2.1]{R}.

\subsection{Generalized derivatives and infinitesimal spaces}

In \cite{GMRV}, a generalization for the derivative of a quasiregular mapping $f$ at $x_0$ was defined as follows. For $t >0$, let
\begin{equation}
\label{eq:fe} 
f_{t}(x) = \frac{ f(x_0  + t x) - f(x_0) }{\rho_f(t)},
\end{equation}
where $\rho_f(r)$ is the mean radius of the image of a sphere of radius $r$ centered at $x_0$ and given by
\begin{equation}
\label{eq:rho} 
\rho_f(r) = \left(\frac{\lambda[f(B(x_0,r))]}{\lambda[B(0,1)]}\right)^{1/n} .
\end{equation}
Here $\lambda$ denotes the standard Lebesgue measure. While each $f_{t}(x)$ is only defined on a ball centered at $0$ of radius $d(x_0,\partial D) / t$, when we consider limits as $t \to 0$, we obtain mappings defined on all of $\R^n$. Of course, there is no reason for such a limit to exist, but because each $f_{t}$ is a quasiregular mapping with the same bound on the distortion, it follows from Theorem \ref{thm:lindist} and Theorem \ref{thm:normal} that for any sequence $t_k \to 0$, there is a subsequence for which we do have local uniform convergence to some non-constant quasiregular mapping.

\begin{definition}
\label{def:genderiv}
Let $f:U \to \R^n$ be a quasiregular mapping defined on a domain $U\subset \R^n$ and let $x_0 \in \R^n$. A generalized derivative $\varphi$ of $f$ at $x_0$ is defined by
\[ \varphi(x) = \lim_{k\to \infty} f_{t_k}(x),\]
for some decreasing sequence $(t_k)_{k=1}^{\infty}$, whenever the limit exists. The collection of generalized derivatives of $f$ at $x_0$ is called the infinitesimal space of $f$ at $x_0$ and is denoted by $T(x_0,f)$.
\end{definition}

To exhibit the behavior of generalized derivatives, we consider some simple examples.

\begin{example}
Let $w\in \C \setminus \{ 0 \}$ and define $f(z) = wz$. Then it is elementary to check that $f_{t}(z) = e^{i\arg w}z$ for any $t >0$. Consequently, $T(0,f)$ consists only of the map $\varphi(z) = e^{i\arg w} z$.
\end{example}

\begin{example}
Let $d\in\N$ and define $g(z) = z^d$. One can check that $f_{t}(z) = z^d$ for any $t>0$ and so $T(0,g)$ consists only of the map $\varphi(z) = z^d$.
\end{example}

These examples illustrate the informal property that generalized derivatives maintain the shape of $f$ near $x_0$, but they lose information on the scale of $f$.
In general, if a quasiregular map $f$ is in fact differentiable at $x_0\in \R^n$, then $T(x_0,f)$ consists only of a scaled multiple of the derivative of $f$ at $x_0$. The reason for the scaling is the use of $\rho_f(r)$ in the definition of $f_{\epsilon}$. We may in fact replace $\rho_f(r)$ by $L_f(x_0,r), l_f(x_0,r)$ or any other quantity comparable to $\rho_f(r)$. In the special case of uniformly quasiregular mappings, that is, quasiregular mappings with a uniform bound on the distortion of the iterates, it was proved in \cite{HMM} that at fixed points with $i(x_0,f) = 1$, they are bi-Lipschitz at $x_0$. Consequently, in this special case one may replace $\rho_f(r)$ with $r$ itself. In general, quasiregular mappings are only locally H\"older continuous and so it does not suffice to use $r$ instead of $\rho_f(r)$.

\begin{definition}
\label{def:simple}
Let $f:U \to \R^n$ be quasiregular on a domain $U$ and let $x_0 \in U$. If the infinitesimal space $T(x_0,f)$ consists of only one element, then $T(x_0,f)$ is called simple.
\end{definition}

In both the examples above, the respective infinitesimal spaces are simple. It was shown in \cite{GMRV} that when the infinitesimal space is simple, then the function is well-behaved near $x_0$. In particular, $f(x) \approx f(x_0) + \rho(|x-x_0|) g((x-x_0)/|x-x_0|)$, where $g :S^{n-1} \to \R^n$ is a function describing the shape of $f$ near $x_0$. Here, $S^{n-1}$ denotes the unit $(n-1)$-dimensional sphere in $\R^n$.

\subsection{Orbit Spaces}
Denote by $C(U,\R^n)$ the set of continuous functions from a domain $U\subset \R^n$ into $\R^n$. If $x\in U$ and $\mathcal{F} \subset C(U,\R^n)$, denote by $E_x:\mathcal{F} \to \R^n$  the point evaluation map, that is, if $f \in \mathcal{F}$ then $E_x(f) = f(x)$.

\begin{definition}
\label{def:orbit}
Let $f:U \to \R^n$ be a quasiregular mapping defined on a domain $U\subset \R^n$ and let $x_0 \in U$. Then the orbit of a point $x\in \R^n$ under the infinitesimal space $T(x_0,f)$ is defined by
\[ \mathcal{O}(x) = E_x(T(x_0,f)) = \{ \varphi(x) : \varphi \in T(x_0,f) \}.\]
\end{definition}

Fletcher and Wallis  show that the orbit space is the accumulation set of a curve.

\begin{theorem}\cite[Theorem 2.10]{FW}
\label{thm:1}
Let $f:U \to \R^n$ be a quasiregular mapping defined on a domain $U\subset \R^n$ and let $x_0 \in U$. Then for any $x\in\R^n$  the orbit space $\mathcal{O}(x)$ is the accumulation set of the curve $t \mapsto f_t(x)$, where $f_t(x)$ is defined by \eqref{eq:fe}.
\end{theorem}

Moreover, \cite[Theorem 1.5]{FN} shows that  for any $x\in\R^n$ we have that $\mathcal{O}(x)$ lies in a ring $\{ y \in \R^n : 1/C' \leq |y| \leq C' \}$ for some constant $C' \geq 1$ depending only on $|x|, n, K_O(f)$ and $i(x_0,f)$.

\begin{corollary}\cite[Corollary 2.11]{FW}
\label{cor:1}
Let $f:U \to \R^n$ be a quasiregular mapping defined on a domain $U\subset \R^n$ and let $x_0 \in U$. Then the infinitesimal space $T(x_0,f)$ either consists of one element or uncountably many.
\end{corollary}

Since Theorem \ref{thm:1} shows that $\mathcal{O}(x)$ is compact and connected and lies in a ring, and finally in dimension two Fletcher and Wallis give the converse statement.

\begin{theorem}\cite[Thereom 2.12]{FW}
\label{thm:2}
Let $X \subset \R^2 \setminus \{ 0 \}$ be a non-empty, compact and connected set. Then there exists a quasiregular mapping $f:\R^2 \to \R^2$ for which $X$ is the image of the point evaluation map $E_{x_1} : T(0,f) \to \R^2$ for $x_1 = (1,0)$.
\end{theorem}

We will prove the converse statement of Theorem \ref{thm:1} for dimension three and higher. 

\begin{theorem}[{\bf Main Result}]
	\label{thm:3}
	Let $X \subset \R^n \setminus \{ 0 \}$ be a non-empty, compact and connected set. Then there exists a quasiconformal mapping $f:\R^n \to \R^n$ for which $X$ is the image of the point evaluation map $E_{x_1} : T(0,f) \to \R^n$ for $x_1 = (1,0,...,0)$.
\end{theorem}

Make note that the point $x_1$ is a choice made as a starting point, and that any point really could have been chosen. For the ease of calculations, $x_1$ is convenient. 

\section{The Zorich Transform}
In this section we will define and discuss some properties of the Zorich Transform, but to do so we must recall the definition of a Zorich map.

\subsection{The Zorich Map}
First let us recall the definition of infinitesimally bilipschitz. 
\begin{definition}
	\label{def:infbilip}
	A function $g:D\to\R^n$, where $D\subset\R^n$, is infinitesimally bilipschitz if there is a constant $L\geq1$ such that \[\frac{1}{L}\leq\liminf_{x\to a}\frac{|g(x)-g(a)|}{|x-a|}\leq\limsup_{x\to a}\frac{|g(x)-g(a)|}{|x-a|}\leq L, \] for all $a\in D$. 
\end{definition}

Note that if we let $a=x\in D$ and $x=x+\epsilon$, with $\epsilon=(\epsilon_1,...,\epsilon_n)$ where  $\epsilon$ is arbitrarily close to the origin, then it is sufficient to show that
\[\frac{1}{L}|\epsilon|=\frac{1}{L}|x-(x+\epsilon)|\leq|g(x)-g(x+\epsilon)|\leq L|x-(x+\epsilon)|=L|\epsilon| \] to satisfy the definition of infinitesimally bilipschitz. 

Here we will first define the class of Zorich maps. Let $g:D\to\R^n$, where $D\subset\R^{n-1}\times\{0\}$ such that $\bar{D}$ is a $(n-1)$-polytope in which under continuous reflection in the $(n-2)$-faces of $\bar{D}$ creates a discrete group. The group $G$ that is acting on $\bar{D}$ is isomorphic to $\Z^{n-1}\times P$, where $P$ is a point group of rotations, see \cite{S}. Also, we must have $g(D)$ is the upper unit sphere (upper in terms of $g(x_1,...,x_{n-1},0)=(y_1,...,y_n)$ is on the unit sphere where $y_n\geq0$) and $g$ is infinitesimally bilipschitz. We can extend the domain $D$ of $g$ to the domain $\R^{n-1}\times\{0\}$  where a reflection in a $(n-2)$-face of $\bar{D}$ in the pre-image corresponds to reflection of the half unit sphere across the $y_1,...,y_{n-1}$-plane so that $y_n\leq 0$. Then as we keep reflecting in the $(n-2)$-faces of the corresponding cells in the pre-image, we appropriately reflect the half unit sphere. Let us denote this extension of $g$ as the function $h: \R^{n-1}\times\{0\}\to\R^n$. 

We define a Zorich Map $\mathcal{Z}:\R^n\to\R^n\setminus\{0\}$ where $n\geq3$  to be \[\mathcal{Z}(x_1,...,x_n)=e^{x_n}h(x_1,...,x_{n-1},0). \]

In particular, Zorich maps defined in this manner are quasiregular. Note that these maps are infinite to one.  Also note that  $\mathcal{Z}$ is strongly automorphic with respect to $G$, see \cite{IM}.

\begin{theorem}
	If  $g:D\to\R^n$, where $D\subset\R^{n-1}\times\{0\}$ such that $\bar{D}$ is a $(n-1)$-polytope as defined above with $g(D)$ being the upper unit sphere and $g$ is infinitesimally bilipschitz, then \[\mathcal{Z}(x)=e^{x_n}h(x_1,...,x_{n-1},0) \] is quasiregular in $\R^n$,  where $h:\R^{n-1}\times\{0\}\to\R^n$ is the extension of $g$ by reflections as defined earlier.
\end{theorem} 
\begin{proof}
	The proof may be found in Appendix A1. 
\end{proof}

For the constructions of our maps we will deal with a particular Zorich map. We will define  
\begin{equation}
\label{eq: gzorich}
g(x_1,...,x_{n-1},0)=\left(\frac{x_1\sin M(x_1,...,x_{n-1})}{\sqrt{x_1^2+\cdots+x_{n-1}^2}},...,\frac{x_{n-1}\sin M(x_1,...,x_{n-1})}{\sqrt{x_1^2+\cdots+x_{n-1}^2}},\cos M(x_1,...,x_{n-1}) \right), 
\end{equation}
where $M(x_1,...,x_{n-1})=\max\{|x_1|,...,|x_{n-1}|\}$, which maps the $[-\pi/2,\pi/2]^{n-1}$ cube to the half unit sphere in $\R^n$ where $y_n\geq0$ in the image. By considering the limit of $g$ as $(x_1,...,x_{n-1},0)$ goes to the origin, we can extend $g$ by continuity so that $g(0,...,0,0)=(0,...,0,1)$. Showing that $g$ is infinitesimally bilipschitz also takes a lot of calculation which may be found in Appendix A2.  From here on out, when we mention $G$ we mean the group isomorphic to $G=\Z^{n-1}\times P$, where $P$ is the appropriate point group of rotations, that acts on $\bar{D}=[-\pi/2,\pi/2]^{n-1}$, as in the definition of the Zorich map above.

From here we can finally define the Zorich Transform.

\subsection{The Zorich Transform}
For a given Zorich map, we define a Zorich transform $\tilde{f}$ for a continuous function $f:\R^n\to\R^n$ to be 
\[\mathcal{Z}\circ\tilde{f}(x)=f\circ \mathcal{Z}(x). \] 
We will discuss the domain of $\tilde{f}$ in a little bit. The difficulty here is  how the Zorich map is defined, every $(n-2)$ face of $\bar{D}\times\R$ will be in the branch set, whereas with the exponential map we have no branch set. Under a Zorich transform, it may be possible for a neighborhood of a point to move partially through one these $(n-2)$ faces causing the neighborhood to split apart due to the Zorich map being defined using reflections. In other words, it is possible for a sequence to converge to a single point in the domain, whereas in the range of the Zorich transformation, the image of the sequence has subsequences that converge to two or more distinct points. However,  choose $C$ to be $\bar{D}\cup \hat{D}$, where $\hat{D}$  is one of the corresponding adjacent reflections of $D$ in one of the $(n-2)$-faces of $\bar{D}$. In $\R^{n-1}$, $C$ is a fundamental set under the group action of $G$. In particular, we have an equivalence relation on $\R^{n-1}$ defined by the group action of $G$ on $C$. In a natural way, we can extend this equivalence relation to an equivalence relation $\sim$ on $\R^n$ by letting the group action $G$ act appropriately on $B=C\times\R$. Note that $B$ is a fundamental set under the appropriate group action of $G$ on $\R^n$. We can define $\mathcal{Z}:\R^n/\sim\to \R^n\setminus\{0\}$, where under the equivalence class, for a fixed $z\in \R$, we identify points on the boundary $\hat{D}\times\{z\}$ to points on the boundary of $\bar{D}\times\{z\}$. As a consequence, open neighborhoods on the boundary of $B$ may seem disconnected when viewed as a set of $\R^n$, but is really connected under the quotient space. In particular, $\mathcal{Z}$ is a homeomorphism from a fundamental set $B$, as an equivalence class, to $\R^n\setminus\{0\}$. Correspondingly, we can define $\mathcal{Z}^{-1}:\R^n\setminus\{0\}\to B$. From here on out, when we discuss a fundamental set $B$, we are treating it as our base for an equivalence class of our quotient space. In particular, $\tilde{f}$ maps from the equivalence class of $B$ to the equivalence class of $B$, for simplicity we will just state $\tilde{f}:B\to B$. 

As a result of "restricting" our attention to $B$ by considering the quotient space, we can see that $\mathcal{Z}\circ\mathcal{Z}^{-1}$ is the identity map. Furthermore, when we have  quasiregular maps $f$ and $g$ we get that 
	\[\tilde{f}\circ\tilde{g}=\mathcal{Z}^{-1}\circ f\circ\mathcal {Z}\circ\mathcal{Z}^{-1}\circ g\circ\mathcal{Z}=\mathcal{Z}^{-1}\circ f\circ g\circ\mathcal{Z}=\widetilde{f\circ g}. \] 

For our particular Zorich map $\mathcal{Z}$ defined by $g$ we choose our fundamental set to be  \[B:=\left(\left(\left[-\frac{\pi}{2},\frac{\pi}{2} \right]\times\left[-\frac{\pi}{2},\frac{\pi}{2} \right]^{n-2}\right)\cup\left(\left(\frac{\pi}{2},\frac{3\pi}{2} \right)\times \left(-\frac{\pi}{2},\frac{\pi}{2} \right)^{n-2} \right) \right) \times\R  \] Then our Zorich transform will be continuous from $B$ to $B$. From here on out, when we reference the Zorich map $\mathcal{Z}$ we mean $\mathcal{Z}:B\to\R^n\setminus\{0\}$ with the corresponding $g$ from \eqref{eq: gzorich}.
We can make note that for a quasiregular map $f$, that since the corresponding function $g$ with the Zorich map $\mathcal{Z}$ is infinitesimally bilipschitz, the distortion of $f$ under conjugation with the Zorich map will still be uniformly bounded from above and vice versa. That is we have
\begin{proposition}
	\label{prop:ZTfqr iff fqr}
	A  map $f:U\to\R^n$ is quasiregular, if and only if the Zorich transform $\tilde{f}=\mathcal{Z}^{-1}\circ f\circ \mathcal{Z} $ is quasiregular, with $\tilde{f}:D\to B$, with $D\subseteq B$. Also, if $f$ is quasiconformal, an injective quasiregular map, from $\R^n$ to $\R^n$, then $\tilde{f}:B\to B$ is also quasiconformal.
\end{proposition}

 In certain cases we can define a Zorich transform globally. For example, if we are in $n\geq3$ dimensions Mayer,\cite{May}, gives us an example where $\tilde{f}$ is multiplication by an integer(note that we are starting with the Zorich Transform first) and then solves the Schr\"oder equation $f\circ h=h\circ \tilde{f}$ by letting  $h$ be the Zorich map which results in giving us a power type map $f$.

Let us look at a couple of more examples when we are in three dimensions. 

\begin{example}
	If $A_{\theta,l}$ is a rotation by $\theta$ about the line $l$ which passes through the origin, then we want to find $\tilde{f}$ such that $\tilde{f}\circ \mathcal{Z}=\mathcal{Z}\circ A_{\theta,l}$. Now, for trivial rotation $A_{\theta}$, where $\theta=2k\pi$, $k\in\Z$, we can define $\tilde{f}$ globally. We can treat $\tilde{f}$ that maps from $B$ to $B$. In this consideration, we are looking at a fixed height $r$ where $\mathcal{Z}$ maps onto the sphere of radius $e^r$. For simplicity of our conversation, we can let $r=0$ so that we are looking at the unit sphere. Suppose $A_{\theta,l}$ was a rotation about the $z$-axis, then the points $(0,0,1)$ and $(0,0,-1)$ are fixed under $A_{\theta,l}$. Also, the unit circle on the $xy$-plane maps onto itself. The pre-image of $(0,0,1)$ and $(0,0,-1)$ under $\mathcal{Z}$ are the center points in the corresponding squares in $B$. The pre-image of the unit circle on the $xy$-plane under $\mathcal{Z}$ is the boundary of the first square. For any circle of radius $s<1$ centered at $(0,0,1)$ or $(0,0,-1)$ on the unit sphere is a corresponding square centered about  the corresponding pre-image of a fixed point.  For $A_{\theta}$ we can define $\tilde{f}$ to be a "rotation" about these squares, so that when we apply $\mathcal{Z}$ we get the image of $A_{\theta,l}$. That is, we get Figure 1: \\
	
	\begin{figure}[h!]
	\includegraphics[scale=0.75]{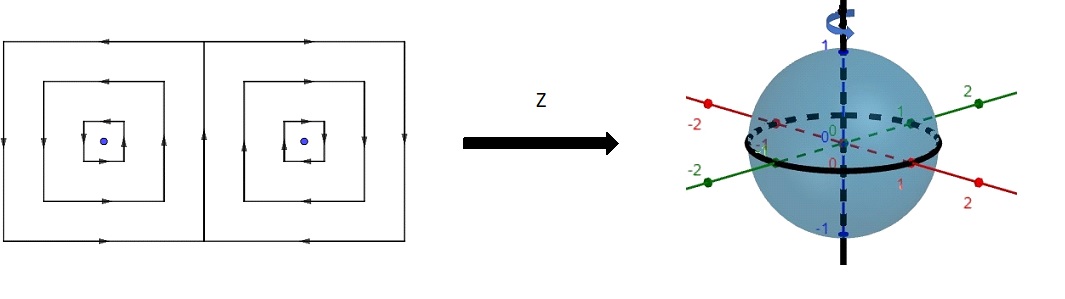}
	\caption[Figure 1]{}
	\end{figure}

	We can also see, that if we were to try to extend $\tilde{f}$ by reflections so that $\tilde{f}:\R^3\to\R^{3}$, then at the slice level our map would no longer be continuous. For example, we have Figure 2,
	\begin{figure}[h!]
		\includegraphics[scale=0.75]{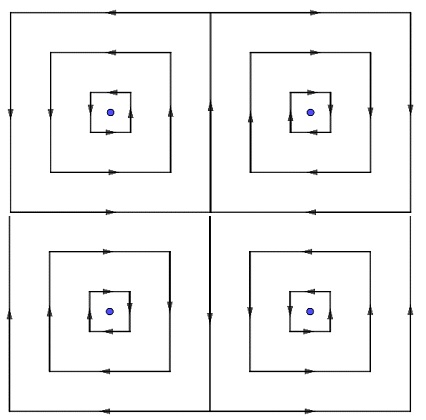}
		\caption[Figure 2]{}
	\end{figure}
	
	where we can see that if we were to place a neighborhood around a corner of four squares, the neighborhood would split into two different directions. We saw earlier that in special cases we can have a Zorich transfrom defined from $\R^n$ to $\R^n$, but this previous example shows that to guarantee continuity of our Zorich transforms we need to restrict the domain and codomain to the fundamental set $B$. \\
	
	Now, suppose $A_{\theta,l}$ is a rotation about a line that is not the $z$-axis. As in the above example, $A_{\theta,l}$ has two fixed points, and a great circle in the unit sphere that maps onto itself. Around the fixed points we can find the pre-images of the circles centered about the fixed points on the circle under $\mathcal{Z}$. Then we can define $\tilde{f}$ as rotations about the pre-images of those circles centered about the fixed points. For example, if we rotated the sphere about a line that goes through two branch points, then the great circle must go through the other two branch points of $\mathcal{Z}$. Then $\tilde{f}$ has the  type of flow map found in Figure 3. 
	
	\begin{figure}[h!]
		\includegraphics[scale=0.25]{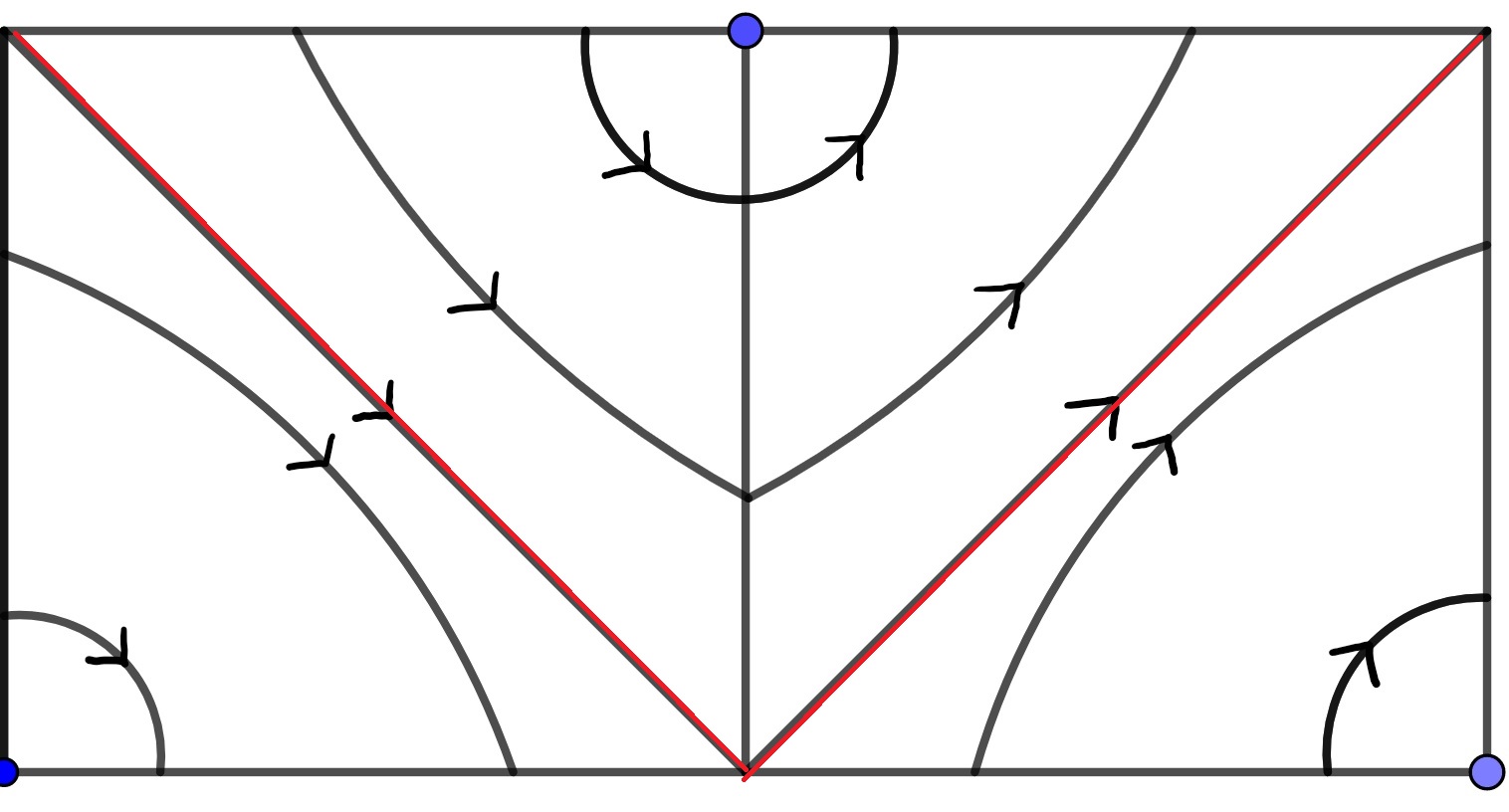}
		\caption[Figure 3]{}
	\end{figure}

In particular, if we have a different rotation we will have a different flow map. Restricting ourselves to the unit sphere, to determine what a flow map will look like, we first need to find the points fixed on the unit circle under the rotation. From here we need to determine the pre-image of the great circle that maps unto itself under the rotation. Then from here determine the pre-image of circles on the unit sphere centered at any of the fixed points.
\end{example}

Now that we have a bit of an understanding of how Zorich transforms behave, let us construct the maps we need to prove Theorem \ref{thm:3}. We will use Zorich transforms to show that these maps are indeed quasiregular, actually they will be injective, so we can even say they are quasiconformal.  

\section{Analogues to Logarithmic Spiral and Stretch Mappings}

The main goal of this section is to show that certain mappings that are constructed are quasiconformal. The strategy is to show that the Zorich transform of the constructed maps are quasiconformal, so that we may conclude that the initial mapping is quasiconformal. In Sections 4.1 and 4.2, we will show quasiconformality by showing condition (ii) is satisfied in Theorem \ref{thm: removability} to conclude quasiconformality on the entire domain. Similarly, in Section 4.4, in areas where the Zorich transform of our map is differentiable, we will show condition (i) is satisfied in Theorem \ref{thm: removability} to conclude that the map on the entire domain is quasiconformal.

\subsection{The Radial Stretch Map} 
Note that we are using coordinates $(x_1,...,x_n)$ for points in the fundamental set $B$ and $(y_1,...,y_n)$ in the image of $\mathcal{Z}$. Also, we will use the convention that if \[g(x_1,...,x_n)=(f(x_1,...,x_n)x_1,...,f(x_1,...,x_n)x_n), \] then we write \[g(x_1,...,x_n)=f(x_1,...,x_n)(x_1,...,x_n). \] As a starting point for the maps we will actually be using to prove our result, we will define a map that stretches a sphere centered at the origin radially onto an ellipsoid centered at the origin by a factor of $K\geq1$ in the $y_n$ direction.  We can do this by the function $R:\R^n\to\R^n$ defined by
\[(y_1,...,y_n)\mapsto\frac{K}{\sqrt{K^2+(1-K^2)\cos^2\varphi}}(y_1,...,y_n), \]
where \[\varphi=\cos^{-1}\left(\frac{y_n}{\sqrt{y_1^2+\cdots+y_n^2}}\right).\]
Consider a fixed $x_n$, so that we have a slice of $B$ at height $x_n$. The image of the slice under $\mathcal{Z}$ is a sphere of radius $y_n=e^{x_n}$. By looking at the image of $\mathcal{Z}$, we can determine 
\[\varphi=\cos^{-1}\left(\frac{e^{x_n}\cos M(x_1,...,x_{n-1})}{\sqrt{\frac{e^{2x_n}x_1^2\sin^2M(x_1,...,x_{n-1})}{x_1^2+\cdots+x_{n-1}^2}+\cdots+\frac{e^{2x_n}x_{n-1}^2\sin^2M(x_1,...,x_{n-1})}{x_1^2+\cdots+x_{n-1}^2}+e^{2x_n}\cos^2M(x_1,...,x_{n-1}) }} \right), \]
which gives us 
\[\varphi=\cos^{-1}\left(\cos M(x_1,...,x_{n-1})\right)=M(x_1,...,x_{n-1}). \]
Then we can define our Zorich transformation $\tilde{R}:B\to B$ by 
\[\tilde{R}(x_1,...,x_n)=\left(x_1,...,x_{n-1},x_n\ln\left(\frac{K}{\sqrt{K^2+(1-K^2)\cos^2M(x_1,...,x_{n-1})}} \right) \right). \]
Notice that under the Zorich transform the first $n-1$ coordinates are unchanged, but the $n$th coordinate rises according to the first $n-1$ coordinates, and as we approach the center of the $n-1$ coordinates we get closer to achieving maximum stretch. 

To define a Zorich transformation for a stretch by $K\geq1$ in any direction, we can conjugate $R$ by a rotation and $\tilde{R}$ by the corresponding Zorich transformation of the rotation. As long as $\tilde{R}$ is quasiconformal, then all the other corresponding maps from conjugation are quasiconformal, we are reduced to the case of looking at $R$. To show that $R$ is quasiconformal, we just need to show that the corresponding Zorich transform $\tilde{R}$ is quasiconformal. Let 
\[V=\ln\left(\frac{K}{\sqrt{K^2+(1-K^2)\cos^2M(x_1,...,x_{n-1})}}\right). \]
For now, we will look at 
\[A_1:=\left\{(x_1,...,x_{n})\in\left[\frac{-\pi}{2},\frac{\pi}{2} \right]^{n-1}\times\{x_n\}:x_1>|x_i|\text{ for }2\leq i\leq n-1 \right\}, \] 
so that $M(x_1,...,x_{n-1})=x_1$. We can compute
\[V_{x_i}=\begin{cases}
\frac{(1-K^2)\cos x_1\sin x_1}{K^2+(1-K^2)\cos^2x_1} & \text{if }i=1\\
0&\text{if }2\leq i\leq n
\end{cases}. \]
Note that $K^2+(1-K^2)\cos^2x_1=1+(K^2-1)\sin^2x_1\in[1,K^2]$ so that 
\begin{equation}
\label{eq:Vx1 bound}
|V_{x_1}|=\left|\frac{(1-K^2)\cos x_1\sin x_1}{K^2+(1-K^2)\cos^2x_1} \right|\leq K^2-1. 
\end{equation}
We are restricting ourselves to the set $A_1$ because when we analyze other pyramid sections of the cube $[-\pi/2,\pi/2]^{n-1}$ the only change would be that $V_{x_1}$ would now be zero and the $i$th location of the derivative will be as above with $x_1$ replaced with $x_i$, which will not change norm calculations of the derivative matrix $\tilde{R}'$ or $(\tilde{R}')^{-1}$
for the regions $A_i$. From here it is a relatively simple calculation to show $\tilde{R}'$ and $(\tilde{R}')^{-1}$ have bounded norm in $A_1$ and all other $A_i$'s. We perform the norm calculations for the radial stretch and radial stretch interpolation map together, which can be found in Section 4.3. Since we have finitely many $A_i$'s and the boundaries of each pyramid section forms a closed set which is a $\sigma$-finite $(n-1)$-dimensional Hausdorff measurable set, Theorem \ref{thm: removability} tells us that $\tilde{R}$ is quasiconformal on $B$. Using this radial stretch map, we will define  the radial stretch interpolation maps. 

\subsection{Radial Stretch Interpolation Maps} 
We want to take a spherical shell and stretch the outer shell by a factor of $K\geq 1$ and the inner shell by a factor of $L\geq 1$ in the same direction, when the inner and outer parts of the shell never cross from the different stretching. We will define a map by stretching in the $y_n$ direction, but we can get any direction by conjugation our function by rotations as before mentioned with radial stretch map. We will let $s,t\in\R$ with $s<t$ such that we are stretching by a factor of $L$ when $y_n=e^{s}$ and a factor of $K$ when $y_n=e^t$, with 
\[|\ln(K/L)|<\frac{t-s}{2}. \]
Letting $\nu=\frac{x_n-s}{t-s}=\frac{\ln|y|-s}{t-s}$, we can define the radial interpolation map to be \begin{equation}
\label{eq:Radint} R_I(y_1,...,y_n)=\left(\frac{K}{\sqrt{K^2+(1-K^2)\cos^2\varphi}}\right)^\nu\left(\frac{L}{\sqrt{L^2+(1-L^2)\cos^2\varphi}} \right)^{1-\nu}\left(y_1,...,y_n\right), \end{equation} with domain \[\{y\in\R^n :e^s\leq|y|\leq e^t\}. \]
We have the corresponding Zorich transform defined on 
\[\{x\in B: s\leq x_n\leq t \} \]
 and is defined by
  \[\tilde{R}_I(x_1,...,x_n)=(x_1,...,x_{n-1},x_n+V_I), \]
where 
\[V_I=\ln\left(\frac{K}{\sqrt{K^2+(1-K^2)\cos^2M'}} \right)\frac{x_n-s}{t-s}+\ln\left(\frac{L}{\sqrt{L^2+(1-L^2)\cos^2M'}} \right)\frac{x_n-t}{s-t}, \] 
with 
\[M'=M(x_1,...,x_{n-1}). \]
As before, we will look at 
\[\left\{(x_1,...,x_{n})\in\left[\frac{-\pi}{2},\frac{\pi}{2} \right]^{n-1}\times\{x_n\}:x_1>|x_i|\text{ for }2\leq i\leq n-1 \right\}, \] 
so that $M(x_1,...,x_{n-1})=x_1$. This gives us 
\[(V_I)_{x_i}=\begin{cases}
\frac{(1-K^2)\cos x_1\sin x_1}{K^2+(1-K^2)\cos^2 x_1}\frac{x_n-s}{t-s}+\frac{(1-L^2)\cos x_1\sin x_1}{L^2+(1-L^2)\cos^2 x_1}\frac{x_n-t}{s-t}& \text{if }i=1\\
0&\text{if }2\leq i\leq n-1\\
\frac{1}{t-s}\ln\left(\frac{K\sqrt{L^2+(1-L^2)\cos^2x_1}}{L\sqrt{K^2+(1-K^2)\cos^2x_1}} \right)&\text{if }i=n
\end{cases}.\]
Note again that $K^2+(1-K^2)\cos^2x_1=1+(K^2-1)\sin^2x_1$ so that by similar calculations as before we have 
\begin{equation}
\label{eq:Vix1 bound}
|(V_I)_{x_1}|\leq K^2+L^2-2, 
\end{equation}
and 
\begin{align*}
|(V_I)_{x_n}|&=\left|\frac{1}{t-s}\ln\left(\frac{K\sqrt{L^2+(1-L^2)\cos^2x_1}}{L\sqrt{K^2+(1-K^2)\cos^2x_1}} \right)\right|\\
&\leq\frac{1}{t-s}\left|\ln\left(\frac{K}{L}\right)\right|\\
&<\frac{1}{t-s}\frac{t-s}{2}=\frac{1}{2},
\end{align*}
so that 
\begin{equation}
\label{eq:Vixn bound}
|(V_I)_{x_n}|<\frac{1}{2}.
\end{equation}
 We can use these calculations to show that the norms of the matrices $\tilde{R}_I'$ and $\left(\tilde{R}_I'\right)^{-1}$ are bounded
 in each pyramid section $A_i$, which may be found in the next section.  Using a similar argument to conclude that $\tilde{R}$ is quasiconformal, we can now  conclude that $\tilde{R}_{I}$ is quasiconformal so that $R_I$ is quasiconformal. 

\subsection{Norm Calculations for $\tilde{R}$ and $\tilde{R_I}$} 
Our goal in this section is to show that the maximal dilatation is bounded. To do so, we will show that$\|\tilde{R}'\|\|(\tilde{R}')^{-1}\|$ and $\|\tilde{R}_I'\|\|(\tilde{R}_I')^{-1}\|$ are bounded above by a real number greater than one.

Let $A$ be either $\tilde{R}'$ or $\tilde{R}_I'$. this means that 
\[A=\begin{pmatrix}
1&0&\cdots&0\\
0&1&\cdots&0\\
\vdots&\ddots&\cdots&\vdots\\
C& 0&\cdots&1+\epsilon
\end{pmatrix}, \]
where 
\[C=\begin{cases}
V_{x_1}&\text{if }A=\tilde{R}'\\
(V_I)_{x_1}&\text{if }A=\tilde{R}_I'
\end{cases}, \]
and
\[\epsilon=\begin{cases}
0&\text{if }A=\tilde{R}'\\
(V_I)_{x_n}&\text{if }A=\tilde{R}_I'
\end{cases}.\]
Using  \eqref{eq:Vx1 bound} and \eqref{eq:Vix1 bound} we have 
\begin{equation}
\label{eq:R'Ri' C bound}
|C|\leq K^2+L^2-2<K^2+L^2.
\end{equation}
From the derivative of $\tilde{R}$ and from \eqref{eq:Vixn bound} we have 
\begin{equation}
\label{eq:R'Ri' epsilon bound}
|\epsilon|<\frac{1}{2}.
\end{equation} 
Also note that \[A^{-1}=\begin{pmatrix}
1&0&\cdots&0\\
0&1&\cdots&0\\
\vdots&\ddots&\cdots&\vdots\\
\frac{-C}{1+\epsilon}&0&\cdots&\frac{1}{1+\epsilon}
\end{pmatrix}. \] 
Using  \eqref{eq:R'Ri' C bound} and \eqref{eq:R'Ri' epsilon bound} we have that 
\begin{align*}
\|A\|&=\sup_{|y|=1}|Ay|=\sup_{|y|=1}\sqrt{y_1^2+\cdots+y_{n-1}^2+(Cy_1+(1+\epsilon)y_n)^2 }\\
&=\sup_{|y|=1}\sqrt{y_1^2+\cdots+y_{n-1}^2+y_{n}^2+C^2y_1^2+2C(1+\epsilon)y_1y_n+\epsilon^2y_n^2 }\\
&\leq \sqrt{1+(K^2+L^2)^2+2(K^2+L^2)+\frac{1}{4} }=\sqrt{\frac{5}{4}+(K^2+L^2)^2+3(K^2+L^2) }=H_1. 
\end{align*}
We also have by Equations  \eqref{eq:R'Ri' C bound} and \eqref{eq:R'Ri' epsilon bound} that 
\begin{align*}
\|A^{-1}\|&=\sup_{|y|=1}|A^{-1}y|=\sup_{|y|=1}\sqrt{y_1^2+\cdots+y_{n-1}^2+\left(\frac{-C}{1+\epsilon}y_1+\frac{1}{1+\epsilon}y_n \right)^2 }\\
&=\sup_{|y|=1}\sqrt{y_1^2+\cdots+y_{n-1}^2+\frac{C^2}{(1+\epsilon)^2}y_1^2-\frac{2C}{(1+\epsilon)^2}y_1y_n+\frac{1}{(1+\epsilon)^2}y_n^2 }\\
&\leq\sup_{|y|=1}\sqrt{y_1^2+\cdots+y_{n-1}^2+(K^2+L^2)^2y_1^2+2(K^2+L^2)|y_1||y_n|+y_n^2 }\\
&\leq\sqrt{1+(K^2+L^2)^2+2(K^2+L^2) }=\sqrt{(1+(K^2+L^2))^2 }\\
&=1+K^2+L^2=H_2. 
\end{align*}
Then we have $H=\|A\|\|A^{-1}\|\leq H_1H_2=H'$. Therefore, $\tilde{R}$ and $\tilde{R}_I$  have bounded maximal dilatation in each pyramid section $A_i$.

\subsection{Radial Stretch with Spiraling Map}

As previously done, we will first show that a specific radial stretch map with spiraling is quasiconformal, and then by conjugation with rotations, or Zorich transforms of rotations, we get that all the other radial stretch with spiraling maps are quasiconformal. We will show that the derivative matrix of the Zorich transform of the radial stretch map with spiraling exists and is bounded from above  in particular regions. We will also show that the Jacobian of the Zorich transform of the radial stretch map with spiraling is positive, i.e. is sense-preserving,   in those same regions.

 First, we define the radial stretch map with spiraling $R_{s}:\R^n\to\R^n$ to be 
\begin{equation}
\label{eq:radsprstr}
R_{s}(y_1,...,y_n)=\frac{K}{\sqrt{K^2+(1-K^2)\frac{y_1^2}{y_1^2+\cdots+y_n^2}}}\left(u,v,y_3,...,y_n \right)  
\end{equation}
with
\begin{align*}
u&=y_1\cos\left(\alpha\ln\sqrt{y_1^2+\cdots+y_n^2 } \right)-y_2\sin\left(\alpha\ln\sqrt{y_1^2+\cdots+y_n^2 } \right)\\
v&=y_1\sin\left(\alpha\ln\sqrt{y_1^2+\cdots+y_n^2 } \right)+y_2\cos\left(\alpha\ln\sqrt{y_1^2+\cdots+y_n^2 } \right),
\end{align*} 
where $\alpha$ is a fixed real number. This map dilates by a factor of $K\geq1$ in the $y_1$ direction while simultaneously rotating in the $y_1,y_2$-plane, creating a spiral. Here on out, we will be looking at the Zorich transform of the above map, $\tilde{R_{s}}:B\to B$ where $B$ is the before mentioned fundamental domain of the Zorich map, to show that $\tilde{R_{s}}$ is quasiconformal and hence $R_{s}$ will also be quasiconformal. Let 
\begin{align*}
M&=M(x_1,...,x_{n-1})=\max\{|x_1|,...,|x_{n-1}| \}\text{ and,}\\
m&=m(x_1,...,x_{n-1})=\min\left\lbrace\frac{1}{|x_1\cos(\alpha x_n)-x_2\sin(\alpha x_n)|},\frac{1}{|x_1\sin(\alpha x_n)+x_2\cos(\alpha x_n)|},\frac{1}{|x_3|},...,\frac{1}{|x_{n-1}|}  \right\rbrace,
\end{align*} 
where $\alpha$ is the same as in the definition of $R_{s}$. The Zorich transform of $R_{s}$ is defined by \[\tilde{R_s}(x_1,...,x_n)=(u_1,...,u_n) \] with \[u_i=\begin{cases}
Mm(x_1\cos(\alpha x_n)-x_2\sin(\alpha x_n))& \text{ for }i=1\\
Mm(x_1\sin(\alpha x_n)+x_2\cos(\alpha x_n))& \text{ for }i=2\\
Mmx_i&\text{ for } 3\leq i\leq n-1\\
x_n+\ln K-\frac{1}{2}\ln\left(K^2+(1-K^2)\frac{x_1^2\sin^2 M}{x_1^2+\cdots+x_{n-1}^2} \right)& \text{ for }i=n
\end{cases}. \]

 We will need to discuss bounding on  $\|\tilde{R_s}'\|$ where the derivative exist. For $x_{1},...,x_{n-1}$ not all zero, \[M(u_1,...,u_{n-1})=M(x_1,...,x_{n-1}), \] since $m$ will cancel with one of the following, $(x_1\cos(\alpha x_n)-x_2\sin(\alpha x_n)),
(x_1\sin(\alpha x_n)+x_2\cos(\alpha x_n)),x_3,...,x_{n-1}$ leaving one of the $u_i$ as $\pm M$. By definition of $m$ we have that $|u_j|\leq|u_i|$ for $1\leq j\leq n-1$. We indeed have that $\mathcal{Z}\circ \tilde{R_s}=R_s\circ \mathcal{Z}$.

To keep $\tilde{R_s}$ injective and sense preserving we can choose $\alpha$ to be sufficiently small so that \[J_{\tilde{R_s}}>2^{-(n+1)/2}.\] 
Before further discussion, recall that $B=C\times\R$. We want $\alpha$ to be small enough so that when we apply $\tilde{R_s}$ the images of $C\times\{z_1\}$ and $C\times\{z_2\}$, where $z_1,z_2\in\R$ and $z_1$ and $z_2$ are close together, do not intersect. This corresponds to the images of two spheres, which are close together, spiral slow enough under $R_s$ so that the images do not intersect.  In particular, our map will remain injective, and hence is a homeomorphism. Also, in Appendix A3, we will split the function into the regions where the mapping is differentiable, we will also note that the regions where we are not differentiable form a closed set that is also a $\sigma$- finite $(n-1)$-dimensional Hausdorff measurable set.  In each region where we are differentiable, we will show that $\|\tilde{R}_s'\|$ is bounded from above. Using this upper bound along with the lower bound for the Jacobian allows us to use Theorem \ref{thm: removability} to conclude that $\tilde{R_s}$ is quasiconformal, and hence $R_s$ is also quasiconformal. To understand the lower bound for the Jacobian and the calculations bounding the the norms of the derivative matrix in all the regions where the derivative exists, please see Appendix A3.   

\section{Realizing the orbit space}

In this section we will prove Theorem \ref{thm:3}, that given a non-empty, compact, connected subset of $\R^n \setminus \{ 0 \}$, we can realize it as an orbit space for a quasiregular, and in fact quasiconformal, map. Before doing so, we will introduce a couple of results that will be necessary. 

Let $f:U\to\R^n$ be a quasiregular mapping defined on $U\subset\R^n$ and let $x_0\in U$. By Theorem \ref{thm:lindist}, we can find $r_0>0$ small enough so that if $0<r<r_0$ then \[\frac{L_f(x_0,r)}{l_f(x_0,r)}\leq C_1, \] where $C_1=2C$ depends only on $n$, $K_O(f)$ and $i(x_0,f)$. For $x\in\R^n$ fixed and $0<t\leq r_0/|x|,$ consider the curve 
\begin{equation}
\label{eq:curve} \gamma_x=\frac{f(x_0+tx)-f(x_0)}{\rho_f(t)}.
\end{equation}
We know that the curve $t\mapsto\gamma_x(t)$ is continuous for $0<t<r_0/|x|$, \cite[Lemma 3.1]{FW}.

Let us define  $h_{(K,\sigma,A)}$ to be a composition where we first stretch radially by a factor of $K$ in the $x_1$ direction using $R$, then followed by a composition of a rotation so that the stretch is in the direction of $\sigma\in S^{n-1}$, and then by an orthogonal map $A$ that fixes the line through $\sigma$ and the origin.  In two dimensions there is a single way to radially stretch by a factor of $K$ in direction $\sigma$, whereas there are many ways to radially stretch by a factor of $K$ in direction $\sigma$ when $n\geq 3$. Whenever we introduce an orthogonal map, it is meant to give us the exact ellipsoid to match with the paths described later on. In particular, $h_{(K,\sigma,A)}$ is the family of all maps that stretch by a factor of $K$ in the $\sigma$ direction. 

\begin{lemma}
	\label{lem: gendivcalc}
	Let $K>0$, $\sigma\in S^{n-1}$, $A$ an orthogonal map that fixes the line through $\sigma$ and the origin, and let $h_{(K,\sigma,A)}$ to be defined as mentioned. Then for $r>0$, we have 
	\begin{equation}
	\frac{h_{(K,\sigma,A)}(rx_1)}{\rho(r)}=K^{1-1/n}\sigma,
	\end{equation}
	where $x_1=(1,0,...,0)\in\R^n$.
\end{lemma}
\begin{proof}
	The volume of the image of a closed ball of radius $r$ under $h_{(K,\sigma,A)}$ is an ellipsoid a semi axis of length $Kr$ and the other semi axes of length $r$. We have that  \[\rho(r)=\left(\frac{\omega_n Kr^n}{\omega_n} \right)^{1/n}=K^{1/n}r, \] where $\omega_n$ is the volume of the unit ball in $\R^{n}$. Therefore, \[\frac{h_{(K,\sigma,A)}(rx_1)}{\rho(r)}=\frac{Kr\sigma}{K^{1/n}r}=K^{1-1/n}\sigma. \]
\end{proof}
When we allow $x_0=0$, $x_1=(1,0,...,0)\in\R^n$, and recalling \eqref{eq:curve},  for any $r>0$ we have \begin{equation}
\label{eq:curveis}
\gamma_{x_1}(rx_1)=K^{1-1/n}\sigma.
 \end{equation}

Let us define some maps that we will be using. First note that we can write any point $x\in\R^n$ as $u\sigma$ where $u>0$ and $\sigma\in S^{n-1}$.  Let $h_{(K,L,\sigma,A)}$ be $R_I$ where we stretch by a factor $K$ and $L$ as described in Section 4.2, but followed by a composition of a rotation so that the stretch is in the direction of $\sigma\in S^{n-1}$, and then by an orthogonal map $A$ that fixes the line through $\sigma$ and the origin. Note that the domain of $R_I$ is 
\[\{x\in\R^n:e^s\leq|x|\leq e^t  \}, \] 
where $t$ and $s$ are constants such that $|\ln(K/L)|<(t-s)/2$. 
Note that if \[x\in\{x\in\R^n: |x|=e^s\} \] then 
$h_{(K,L,\sigma, A)}(x)=h_{(L,\sigma,A)}(x)$, and if \[x\in\{x\in\R^n: |x|=e^s\} \] then $h_{(K,L,\sigma, A)}(x)=h_{(K,\sigma,A)}(x)$. For the sphere of radius  $\lvert x\rvert\in(e^s,e^t)$ centered at the origin, we have that the image of the sphere is an ellipsoid like shape but not necessarily an ellipsoid.  Let $g_{(K,\sigma_1,\sigma_2, A, B)}$ be $R_s$ where we stretch by a factor of $K$, then composed with  an orthogonal map $A$, which will match us with the ellipsoid corresponding to  $h_{(K,\sigma_1,A)}$, followed by another orthogonal map $B$ so that we start at a point on the radial line through $\sigma_1$ and end our spiraling at a point on the radial line through $\sigma_2$. Once we finish the rotation, we want the image of the map $g_{(K,\sigma_1,\sigma_2, A, B)} $ to correspond to the image of an ellipsoid corresponding to $h_{(K,\sigma_2, A')}$, where $A'$ is the corresponding orthogonal map that matches the directions we want the ellipsoid "turned" about a line through the origin in the $\sigma_2$ direction. Also, in the function $g_{(K,\sigma_1,\sigma_2, A, B)}$, $B$ will counteract $A$ and "turn" the ellipsoid about the radial line so that the spiraling is occurring in the direction we desire. Note that the image of any sphere of radius $r>0$ under $g_{(K,\sigma_1,\sigma_2,A,B)}$ is an ellipsoid by construction.
\begin{proof}[Proof of Theorem \ref{thm:3}]

Let $X \subset \R^n \setminus \{0 \}$ be compact and connected. For $k\in \N$, let $U_k$ be an open $1/k$-neighborhood of $X$. We can find $K\in \N$ and $C>1$ so that for $k\geq K$, $U_k \subset \{ x : 1/C \leq |x| \leq C \}$.
For $k\geq K$, find a path $\Gamma_k \subset U_k$ starting and ending at (possibly different) points of $X$ so that:
\begin{itemize}
\item $\Gamma_k$ is made up of finitely many radial line segments and arcs of great circles,
\item for every $x\in U_k$, there exists $u\in \Gamma_k$ with $|x-u| <1/k$,
\item the endpoint of $\Gamma_k$ coincides with the starting point of $\Gamma_{k+1}$.
\end{itemize}

Our aim is to construct a quasiconformal map $f$ so that, recalling \eqref{eq:curve}, the curve $\gamma_{x_1}$ is the concatenation of $\Gamma_k$ for $k\geq K$. If this is so, then since by construction $\gamma_{x_1}$ accumulates exactly on $X$, we are done.  In the parts of  $f$ that take on $g_{(K,\sigma_1,\sigma_2,A,B)}$, $f$ will send a ball of radius $r$ to an ellipsoid centered at the origin with appropriate eccentricity and orientation, and in the parts of $f$ that take on $h_{(K,L,\sigma,A)}$ $f$ will send a ball of radius $r$ to an ellipsoid centered at the origin with appropriate eccentricity and orientation at least on the boundary of a spherical shell, so that $\gamma_{x_1}(rx_1)$ has the required value. Note, if we have a spherical shell with outer radius $r_k>0$ and inner radius $r_{K+1}$, then for points on a radial line segment between the two boundaries will create a radial line segment under the generalized derivative of $h_{(K,L,\sigma, A)}$ even though the image of spheres in the interior of the spherical shell under $h_{(K,L,\sigma, A)}$ may not be an ellipsoid. Recall that Lemma \ref{lem: gendivcalc} and \eqref{eq:curveis} says what ellipsoid we need to obtain a required value for $\gamma_{x_1}(rx_1)$.

To this end, we will give a parameterization $p_k : [r_{k+1},r_k] \to \Gamma_k$ for $k\geq K$, where $r_k$ is given and $r_{k+1}$ is to be determined, with the requirements that $r_{k+1}<r_k$ and $r_k\to0$ as $k\to\infty$.
Suppose $k\geq K$, we have the open set $U_k$ and a point $p_k(r_k) \in X$. We can find a path $\Gamma_k$ with the required properties, made up of $\Gamma_k^1, \ldots, \Gamma_k^{m}$ where $m=m(k)$ and each $\Gamma_k^j$ is either a radial line segment or an arc of a great circle. We must have $r_k^m = r_{k+1}^1$. The parameterization for $\Gamma_k^j$ is given by $p_k^j:[r_k^{j+1},r_k^j]\to \Gamma_k^j$, where we are given $r_k^j$ and have to determine $r_k^{j+1}$.

{\bf Case (i):} $\Gamma_k^j$ is an arc of a great circle, say from $u\sigma_1$ to $u\sigma_2$ with $1/C \leq u \leq C$ and the appropriate orientation. By \eqref{eq:curveis} and our earlier discussion, on $|x| = r_k^j$ we have $f(x) = h_{(u^{n/(n-1)}),\sigma_1,A}(x)$ and $\gamma_{x_1}(r_k^jx_1) = u\sigma_1$.\\

From Section 4.4, we can let  $K = u^{n/(n-1)}$ and $\alpha$ chosen with parity to give the correct direction of spiraling commensurate with the orientation of our piece of great circle, and $|\alpha|$ chosen small enough so that $J_{g_{(K,\sigma_1,\sigma_2, A, B)}}$ is bounded from below, by $2^{-(n+1)/2}$. We then choose $r_k^{j+1}$ so that on $\{ x:  r_k^{j+1} \leq  |x| \leq r_k^j \}$,
\[ f(x) = r_k^j g_{(K,\sigma_1,\sigma_2, A, B)}\left(\frac{x}{r_k^j}\right),\]
and $f(r_k^{j+1}x_1) = u^{n/(n-1)}\sigma_2$. Recall that $B$ is the orthogonal map chosen that will guarantee that we are spiraling in the correct direction. Then by \eqref{eq:curveis} and earlier discussion, we have $\gamma_{x_1}(r_k^{j+1}x_1) = u\sigma_2$.  Also note that we an choose $\alpha$ small enough so that $f$ has bounded distortion of a constant depending on $C$, by construction of $g_{(K,\sigma_1,\sigma_2,A, B)}$. 

{\bf Case (ii):} $\Gamma_k^j$ is a radial  line segment, say from $u_1\sigma$ to $u_2\sigma$ with $u_1,u_2 \in [1/C,C]$.
 By \eqref{eq:curveis} and earlier discussion, on $|x| = r_k^j$ we have $f(x) = h_{(u_1^{n/(n-1)},\sigma,A)}(x)$ and $\gamma_{x_1}(r_k^jx_1) = u_1\sigma$.

Looking back at our discussion in Section 4.2, we can let $K=u_1^{n/(n-1)}$ and $L=u_2^{n/(n-1)}$, and we can choose $s$ and $t$ so that $h_{(K,L,\sigma,A)}$ is quasiconformal. Choosing $r_k^{j+1}=(r_k^je^s)/e^t,$ we have
\[ f(x) =\frac{r_k^j}{e^t} h_{K,L,\sigma, A}\left(\frac{xe^t}{r_k^j}\right),\]
with $f(r_k^{j+1}x_1) = u_2^{n/(n-1)}\sigma$. Then by \eqref{eq:curveis} and earlier discussion, we have $\gamma_{x_1}(r_k^{j+1}x_1) = u_2\sigma$. Also, we have chosen $s$ and $t$ so that the distortion depends on a constant in terms of $C$.

These two cases show how to parameterize each sub-arc of $\Gamma_k$ and hence inductively how to define a parameterization for $\gamma_{x_1}$ from $(0,r_K]$. By construction, the obtained map $f$ has uniformly bounded distortion  and hence is quasiconformal. 
\end{proof}

\section{Appendix}

\subsection{Appendix A1: Calculations to Show that the Zorich Map is Quasiregular}

\begin{theorem}
	If  $g:D\to\R^n$, where $D\subset\R^{n-1}\times\{0\}$ is a $n-1$ regular polytope with $g(D)$ being the upper unit sphere is infinitesimally bilipschitz, then \[\mathcal{Z}(x)=e^{x_n}h(x_1,...,x_{n-1},0) \] is quasiregular in $\R^n$,  where $h:\R^{n-1}\times\{0\}\to\R^n$ is the extension of $g$ by reflections as defined earlier.
\end{theorem} 
\begin{proof}
	Since $h$ is extended by reflections in $(n-2)$-faces of $D$, we can restrict our attention to $h|_D=g$. Note that we can see that since $g$ is infinitesimally bilipschitz that \[\mathcal{Z}|_D(x)=e^{x_n}g(x_1,...,x_{n-1},0) \] is absolutely continuous on lines. Also, since we are multiplying each coordinate in the image of $g$ by $e^{x_n}$ we can see that $\mathcal{Z}|_D$ must also be locally $L^n$-integrable. All that is left to show is that $\mathcal{Z}_g$ has bounded distortion

	Since $g$ is infinitesimally bilipschitz, there is a $L\geq 1$ such that \[\frac{1}{L}\leq\liminf_{\epsilon\to0}\frac{|g(x+\epsilon)-g(x)|}{|\epsilon|}\leq\limsup_{\epsilon\to0}\frac{|g(x+\epsilon)-g(x)|}{|\epsilon|}\leq L, \] 
for all $x\in D$, $\epsilon=(\epsilon_1,...,\epsilon_n)$. The linear distortion function from Iwaniec and Martin, \cite[Section 6.4]{IM}, of $\mathcal{Z}$ is defined to be 
\begin{align*}
H(x,\mathcal{Z})&=\limsup_{r\to0}\frac{\max_{|\epsilon|=r}|\mathcal{Z}(x+\epsilon)-\mathcal{Z}(x)|}{\min_{|\epsilon|=r}|\mathcal{Z}(x+\epsilon)-\mathcal{Z}(x)|}\\
&=\limsup_{r\to0}\frac{\max_{|\epsilon|=r}|e^{x_n}\left(e^{\epsilon_n}g(x_1+\epsilon_1,...,x_{n-1}+\epsilon_{n-1},0)-g(x_1,...,x_{n-1},0)\right)|}{\min_{|\epsilon|=r}|e^{x_n}\left(e^{\epsilon_n}g(x_1+\epsilon_1,...,x_{n-1}+\epsilon_{n-1},0)-g(x_1,...,x_{n-1},0)\right)|}. 	
\end{align*}

Note that \[\lim_{x\to0}\frac{e^x-1}{x}=1, \] so there is  $a>0, a\in\R$, such that $|e^{\epsilon_n}-1|=a|\epsilon_n|$, where $a\to1$ as $\epsilon_n\to0$. For notation, let $\bar{x}=(x_1,...,x_{n-1},0)$. Also note that $r^2=|\epsilon|^2=|\bar{\epsilon}|^2+|\epsilon_n|^2$, so that $|\epsilon_n|^2=r^2-|\bar{\epsilon}|^2$. This leads to 
\begin{align*}
|\mathcal{Z}(x+\epsilon)-\mathcal{Z}(x)|&=e^{x_n}|e^{\epsilon_n}g(\bar{x}+\bar{\epsilon})-g(\bar{x})|\\
&=e^{x_n}|e^{\epsilon_n}\left(g(\bar{x}+\bar{\epsilon})-g(\bar{x}) \right)+g(\bar{x})\left(e^{\epsilon_n}-1 \right) |.
\end{align*}
Notice that $g(\bar{x}+\bar{\epsilon})-g(\bar{x})$ describes how the first $n-1$ coordinates map onto the unit sphere. In particular for a point $A$ on the unit sphere, $g(\bar{x}+\bar{\epsilon})-g(\bar{x})$ moves point $A$ to point $B$, still on the unit sphere, by a distance of $c|\bar{\epsilon}|$ where $\frac{1}{L}\leq c\leq L$, since $h$ is bilipschitz. Then $e^{\epsilon_n}-1$ will move point $B$ orthogonally from the unit sphere to a point $C$ by a distance of $|e^{\epsilon_n}-1|=a|\epsilon_n|$. Let $L'$ be the distance from point $A$ to point $C$, in particular \[L'=|e^{\epsilon_n}\left(g(\bar{x}+\bar{\epsilon})-g(\bar{x}) \right)+g(\bar{x})\left(e^{\epsilon_n}-1 \right) |. \] One can also notice that $\angle ABC=\pi/2+\delta$ with $\delta>0$ where $\delta\to0$ as $r\to0$. The linear distance $L'$ is then 
\begin{align*}
L'^2&=\left(c|\bar{\epsilon}| \right)^2+a^2|\epsilon_n|^2-2ac|\bar{\epsilon}||\epsilon_n|\cos\left(\frac{\pi}{2}+\delta\right)\\
&=c^2|\bar{\epsilon}|^2+a^2r^2-a^2|\bar{\epsilon}|^2+2ac|\bar{\epsilon}||\epsilon_n|\delta.
\end{align*}
Since $c\leq L$, we have that \[L'^2\leq L^2|\bar{\epsilon}|^2+a^2r^2+2L|\bar{\epsilon}||\epsilon_n|\delta\leq r^2(L^2+a^2+2La\delta). \] 
For $\epsilon$ sufficiently small, we can have $a$ close enough to $1$ and $\delta $ small enough so that \[L'\leq r\sqrt{L^2+a^2+2La\delta }\leq 2r\sqrt{L^2+1}. \] 
We also have that 
\begin{align*}
L'^2&\geq \frac{1}{L^2}|\bar{\epsilon}|^2+a^2r^2-a^2|\bar{\epsilon}|^2+\frac{2}{L}a|\bar{\epsilon}||\epsilon_n|\delta \\
&\geq \frac{r^2}{L^2}-\frac{|\epsilon_n|^2}{L^2}+a^2|\epsilon_n|^2\\
&=\frac{r^2}{L^2}+\frac{(a^2L^2-1)|\epsilon_n|^2}{L^2}.
\end{align*}
Since $L^2\geq 1$ we have $a^2L^2\geq a^2$ which gives us $a^2L^2-1\geq a^2-1$. We have that \[L'^2\geq\frac{r^2}{L^2}+\frac{(a^2-1)|\epsilon_n|^2}{L^2}.  \] 
Since $r^2=|\bar{\epsilon}|^2+|\epsilon_n|^2$, we know that $|\epsilon_n|\in[0,r]$. If $a^2-1\geq 0$, then \[L'^2\geq \frac{r^2}{L^2} \] which means \[L'\geq \frac{r}{L}>\frac{r}{2\sqrt{L^2+1}}. \] 
If $a^2-1<0$ we have 
\[L'^2\geq\frac{r^2}{L^2}+\frac{(a^2-1)r^2}{L^2}=\frac{a^2r^2}{L^2}.\] 
Since $a\to 1$ as $r\to 0$, we can find $r$ small enough so that $a>\frac{1}{2}$. Then we have \[L'^2\geq\frac{(1/2)^2r^2}{L^2}=\frac{1}{4L^2}. \]  
Again, we get \[L'\geq\frac{r}{2L} >\frac{r}{2\sqrt{L^2+1}}. \] 
Then we have that our linear distortion 
\begin{align*}
H(x,\mathcal{Z})&=\limsup_{r\to0}\frac{\max_{|\epsilon|=r}|e^{x_n}\left(e^{\epsilon_n}g(x_1+\epsilon_1,...,x_{n-1}+\epsilon_{n-1},0)-g(x_1,...,x_{n-1},0)\right)|}{\min_{|\epsilon|=r}|e^{x_n}\left(e^{\epsilon_n}g(x_1+\epsilon_1,...,x_{n-1}+\epsilon_{n-1},0)-g(x_1,...,x_{n-1},0)\right)|}\\
&\leq\frac{2r\sqrt{L^2+1}}{\frac{r}{2\sqrt{L^2+1}}}=4(L^2+1)\leq 8L^2.
\end{align*}

From \eqref{defeq:lindistort} and \eqref{eq: kdistboundbyhlinear} we have that the distortion $K$ of $\mathcal{Z}$ is bounded by 
\[ \left(H(x,\mathcal{Z}) \right)^{n-1}=(8L^2)^{n-1}. \] 
Therefore $\mathcal{Z}$ is quasiregular.  

\end{proof}

\subsection{Appendix A2: Our Particular Function For Zorich Map is Infinitesimally Bilipschitz}
	Recall that we defined 
	\[g(x_1,...,x_{n-1},0)=\left(\frac{x_1\sin M(x_1,...,x_{n-1})}{\sqrt{x_1^2+\cdots+x_{n-1}^2}},...,\frac{x_{n-1}\sin M(x_1,...,x_{n-1})}{\sqrt{x_1^2+\cdots+x_{n-1}^2}},\cos M(x_1,...,x_{n-1}) \right), \]
	where $M(x_1,...,x_{n-1})=\max\{|x_1|,...,|x_{n-1}|\}$, which maps the $[-\pi/2,\pi/2]^{n-1}$ cube to the half unit sphere in $\R^n$ where $y_n\geq0$ in the image.
	The calculations for $n>3$ are very similar, but even more tedious than the calculations for $n=3$. We will show that for $n=3$ that for $g:[-\pi/2,\pi/2]^2\to\R^3$ defined by 
	\[g(x,y,0)=\left(\frac{x\sin M(x,y)}{\sqrt{x^2+y^2}},\frac{y\sin M(x,y)}{\sqrt{x^2+y^2}},\cos M(x,y) \right) \]
	where $M(x,y)=\max\{|x|,|y| \},$ is infinitesimally bilipschitz, and then note that by similarity we can conclude that all other $g$ functions for $n>3$ are also infinitesimally bilipschitz. 
	
	Without loss of generality, since $g$ is symmetric in the square, we will restrict ourselves to \[A:=\{(x,y,z)\in[-\pi/2,\pi/2]^2\times\{0 \}:x\geq |y|\}, \]
	so that $M(x,y)=x$ for $(x,y,z)\in A$. Note that when we take $(x,y,z)\in A$ we can omit the origin, a single point has Lebesgue measure zero, and so our  map will still have bounded distortion and will be quasiregular. First we will note some useful Taylor series expansions: 
	\begin{align*}
	\cos\epsilon&=1-\frac{\epsilon^2}{2}+o(\epsilon^2),\\
	\sin\epsilon&=\epsilon+o(\epsilon^2),\text{ and }\\
	\left((x+\epsilon)^2+(y+\delta)^2\right)^{-1/2}&=(x^2+y^2)^{-1/2}\left(1-\frac{\epsilon x+\delta y}{x^2+y^2}+o(|(\epsilon,\delta)|^2) \right).
	\end{align*}
	Here we will take $\epsilon$ and $\delta$ to be small enough so that $(x+\epsilon,y+\delta)\in A$ for our calculations. One can ask about how we handle the distortion about the boundary of $A$. The following calculations will be similar with same final estimates when we consider the other triangle quadrants, which will give us our infinitesimally bilipschitz result for $h$. We have
	\[|g(x,y)-g(x+\epsilon,y+\delta)|^2=|(u,v,w)|^2, \] 
	where 
	\begin{align*}
	u&=\frac{x\sin x}{\sqrt{x^2+y^2}}-\frac{(x+\epsilon)\sin(x+\epsilon)}{\sqrt{(x+\epsilon)^2+(y+\delta)^2}},\\
	v&=\frac{y\sin x}{\sqrt{x^2+y^2}}-\frac{(y+\delta)\sin(x+\epsilon)}{\sqrt{(x+\epsilon)^2+(y+\delta)^2}},\text{ and}\\
	w&=\cos x-\cos(x+\epsilon).
	\end{align*}
	
	Using the Taylor series above, we have the following calculations, 
	\begin{align*}
	u^2&=\left(\frac{x\sin x}{\sqrt{x^2+y^2}}-\frac{(x+\epsilon)(\sin x\cos\epsilon+\cos x\sin\epsilon)}{\sqrt{x^2+y^2}}\left(1-\frac{\epsilon x+\delta y}{x^2+y^2} \right) \right)^2+o(|(\epsilon,\delta)|^2)\\
	&=\frac{1}{x^2+y^2}\left(\frac{-\epsilon x^2\sin x}{x^2+y^2}-\frac{\delta yx\sin x}{x^2+y^2}+\epsilon x\cos x+\epsilon\sin x \right)^2+ o(|(\epsilon,\delta)|^2)\\
	&=\frac{1}{x^2+y^2}\left(\frac{\epsilon^2x^4\sin^2x}{(x^2+y^2)^2}+\frac{2\epsilon\delta x^3y\sin^2x}{(x^2+y^2)^2}-\frac{2\epsilon^2x^3\sin x\cos x}{x^2+y^2}-\frac{2\epsilon^2x^2\sin^2x}{x^2+y^2}+\frac{\delta^2x^2y^2\sin^2x}{(x^2+y^2)^2} \right)\\
	&+\frac{1}{x^2+y^2}\left(\frac{-2\epsilon\delta x^2y\sin x\cos x}{x^2+y^2}-\frac{2\epsilon\delta xy\sin^2x}{x^2+y^2}+\epsilon^2x^2\cos^2x+2\epsilon^2x\sin x\cos x+\epsilon^2\sin^2x \right)+o(|(\epsilon,\delta)|^2),
	\end{align*}
	
	\begin{align*}
	v^2&=\left(\frac{y\sin x}{\sqrt{x^2+y^2}}-\frac{(y+\delta)(\sin x\cos\epsilon+\cos x\sin\epsilon)}{\sqrt{x^2+y^2}}\left(1-\frac{\epsilon x+\delta y}{x^2+y^2}\right) \right)^2+o(|(\epsilon,\delta)|^2)\\
	&=\frac{1}{x^2+y^2}\left(\frac{-\epsilon xy\sin x}{x^2+y^2}-\frac{\delta y^2\sin x}{x^2+y^2}+\epsilon y\cos x+\delta\sin x \right)^2+o(|(\epsilon,\delta)|^2)\\
	&=\frac{1}{x^2+y^2}\left(\frac{\epsilon^2x^2y^2\sin^2x}{(x^2+y^2)^2}+\frac{2\epsilon\delta xy^3\sin^2x}{(x^2+y^2)^2}-\frac{2\epsilon^2xy^2\sin x\cos x}{x^2+y^2}-\frac{2\epsilon\delta xy\sin^2x}{x^2+y^2}+\frac{\delta^2y^4\sin^2x}{(x^2+y^2)^2} \right)\\
	&+\frac{1}{x^2+y^2}\left(\frac{-2\epsilon\delta y^3\sin x\cos x}{x^2+y^2}-\frac{2\delta^2y^2\sin^2x}{x^2+y^2}+\epsilon^2y^2\cos^2x+2\epsilon\delta y\sin x\cos x+\delta^2\sin^2x \right)+o(|(\epsilon,\delta)|^2),
	\end{align*}
	and
	\begin{align*}
	w^2&=\left(\cos x-(\cos x\cos\epsilon-\sin\epsilon\sin x) \right)^2+o(|(\epsilon,\delta)|^2)\\
	&=\left(\cos x-\cos x+\frac{\epsilon^2}{2}\cos x+\epsilon\sin x\right)^2+o(|(\epsilon,\delta)|^2)\\
	&=\epsilon^2\sin^2x+o(|(\epsilon,\delta)|^2).
	\end{align*}
	
	Then separating into $\epsilon^2$, $\delta^2$, and $\epsilon\delta$ terms we have
	
	\begin{align*}
	|(u,v,w)|^2&=u^2+v^2+w^2\\
	&=\epsilon^2+\frac{\epsilon^2\sin^2x}{(x^2+y^2)^3}\left(x^4-2x^4-2x^2y^2+2x^2y^2+y^4+x^2y^2 \right)\\
	&+\frac{\delta^2\sin^2x}{(x^2+y^2)^3}\left(x^2y^2+y^4-2x^2y^2-2y^4+x^4+2x^2y^2+y^4 \right)\\
	&+\frac{\epsilon\delta\sin^2x}{(x^2+y^2)^3}\left(2x^3y-2x^3y-2xy^3+2xy^3-2x^3y-2xy^3 \right)+o(|(\epsilon,\delta)|^2)\\
	&=\epsilon^2\left(1+\frac{y^2\sin^2x}{(x^2+y^2)^2} \right)+\delta^2\left(\frac{x^2\sin^2x}{(x^2+y^2)^2} \right)-2\epsilon\delta\left(\frac{xy\sin^2x}{(x^2+y^2)^2}\right)+o(|(\epsilon,\delta)|^2).
	\end{align*}
	Here we have that \[|g(x,y)-g(x+\epsilon,y+\delta)|^2=\epsilon^2\left(1+\frac{y^2\sin^2x}{(x^2+y^2)^2} \right)+\delta^2\left(\frac{x^2\sin^2x}{(x^2+y^2)^2} \right)-2\epsilon\delta\left(\frac{xy\sin^2x}{(x^2+y^2)^2}\right)+o(|(\epsilon,\delta)|^2). \]
	
	We can notice that the term \[\epsilon^2\left(1+\frac{y^2\sin^2x}{(x^2+y^2)^2} \right)+\delta^2\left(\frac{x^2\sin^2x}{(x^2+y^2)^2} \right)-2\epsilon\delta\left(\frac{xy\sin^2x}{(x^2+y^2)^2}\right)\] is a quadratic form in $(\epsilon, \delta)$ with corresponding matrix \[B=\begin{pmatrix}
	1+\frac{y^2\sin^2x}{(x^2+y^2)^2}&\frac{-xy\sin^2x}{(x^2+y^2)^2}\\
	\frac{-xy\sin^2x}{(x^2+y^2)^2}& \frac{x^2\sin^2x}{(x^2+y^2)^2}
	\end{pmatrix}. \]
	
	Since we are in quadratic form the eigen-values and -vectors tell us how much and in what direction we have distortion. If the eigen-values are bounded above and below by positive constants, then we have that our map $h$ is infinitesimally bilipschitz. That is, if the eigen-values $\lambda$ have the bounding $\frac{1}{L}\leq\lambda\leq L$ for some $L\geq 1$, then we have \[\frac{1}{L}(\epsilon^2+\delta^2)\leq\epsilon^2\left(1+\frac{y^2\sin^2x}{(x^2+y^2)^2} \right)+2\epsilon\delta\left(\frac{-xy\sin^2x}{(x^2+y^2)^2}\right)+\delta^2\left(\frac{x^2\sin^2x}{(x^2+y^2)^2}\right)\leq L(\epsilon^2+\delta^2), \] which when we consider the small error term gives us an $\tilde{L}\geq 1$ such that \[\frac{1}{\tilde{L}}(\epsilon^2+\delta^2)\leq|g(x,y)-g(x+\epsilon,y+\delta)|^2\leq\tilde{L}(\epsilon^2+\delta^2). \]
	
	To find our eigen-values, we have 
	\begin{align*}
	\det(\lambda I-B)&=\lambda^2-\lambda\left(\frac{x^2\sin^2x}{(x^2+y^2)^2}+1+\frac{y^2\sin^2x}{(x^2+y^2)^2} \right)-\frac{x^2y^2\sin^4x}{(x^2+y^2)^4}+\left(1+\frac{y^2\sin^2x}{(x^2+y^2)^2} \right)\left(\frac{x^2\sin^2x}{(x^2+y^2)^2} \right)\\
	&=\lambda^2-\lambda\left(1+\frac{\sin^2x}{x^2+y^2}\right)+\frac{x^2\sin^2x}{(x^2+y^2)^2}, 
	\end{align*}
	so that $\det(\lambda I-B)=0$ when 
	\[\lambda=\frac{1}{2}\left(1+\frac{\sin^2x}{x^2+y^2}\pm\sqrt{\frac{(x^2-\sin^2x)^2+2x^2y^2+y^4+2y^2\sin^2x}{(x^2+y^2)^2} } \right). \]
	For the rest of the calculations, we will use facts about $\sin x/x$, that is 
	\begin{lemma}
		If $f(x)=\frac{\sin x}{x}$, then $f$ is decreasing on $(0,\pi/2)$ and $f:[0,\pi/2]\to[2/\pi,1].$
	\end{lemma}
	
	Here, we will show that $\lambda>0$, first we are assuming that $x\neq0$, so that we are not at the origin. Also note that 
	\[\lambda=\frac{1}{2a}(-b\pm\sqrt{b^2+4ac})=\frac{1}{2}\left(-b\pm\sqrt{b^2-4c} \right) \] 
	where 
	\begin{align*}
	a&=1,\\
	b&=-\left(1+\frac{\sin^2x}{x^2+y^2} \right),\text{ and}\\
	c&=\frac{x^2\sin^2x}{(x^2+y^2)^2}.
	\end{align*}
	First note that $-b>0$ and $c>0$. We also have \[b^2-4c=\frac{(x^2-\sin^2x)^2+2x^2y^2+y^4+2y^2\sin^2x}{(x^2+y^2)^2}>0, \]
	then $|b|>\sqrt{b^2-4c}.$ This gives us 
	\[\lambda=-b\pm\sqrt{b^2-4c}>0. \]
	Since we have $x\geq|y|$, with $x\neq 0$ since we are not at the origin, then
	\begin{align*}
	\lambda&\leq\frac{1}{2}\left(1+\frac{\sin^2x}{x^2}+\sqrt{\frac{(x^2)^2}{x^4}+\frac{2x^4}{x^4}+\frac{x^4}{x^4}+\frac{2x^2\sin^2x}{x^4} } \right)\\
	&\leq\frac{1}{2}\left(1+1+\sqrt{1+2+1+2}\right)=1+\frac{\sqrt{6}}{2}. 
	\end{align*}
	Let $p=-b$ and $q=\sqrt{b^2-4c}$, then $\lambda=p\pm q$. We showed that 
	\[\lambda\leq p+q<1+\frac{\sqrt{6}}{2}. \]
	Also, note 
	\begin{align*}
	p^2-q^2&=(-b)^2-\left(\sqrt{b^2-4c} \right)^2=b^2-b^2+4c=4c\\
	&=4\frac{x^2\sin^2x}{(x^2+y^2)^2}\geq4\frac{x^2\sin^2x}{(x^2+x^2)^2}=4\frac{x^2\sin^2x}{4x^4}\\
	&=\frac{\sin^2x}{x^2}\geq \frac{4}{\pi^2}, 
	\end{align*}
	since we have $(x,y)\in A$. Here we want to show that $\lambda\geq p-q$ is bounded from below. We know that 
	\begin{align*}
	p+q&\leq 1+\frac{\sqrt{6}}{2},\text{ and} \\
	p-q&\geq\frac{4}{\pi^2}.
	\end{align*}
	This leads to the following calculation, 
	\begin{align*}
	\lambda&\geq p-q=\frac{p^2-q^2}{p+q}\geq\frac{\frac{4}{\pi^2} }{p+q}\\
	&\geq\frac{\frac{4}{\pi^2} }{1+\frac{\sqrt{6}}{2}}=\frac{8}{\pi^2(2+\sqrt{6})}.
	\end{align*}
	Since $\frac{\pi^2(2+\sqrt{6})}{8}>1+\frac{\sqrt{6}}{2}$, then we can let $L=\frac{\pi^2(2+\sqrt{6})}{8}$, so that $g(x,y)$ is infinitesimally bilipschitz. 
	
\subsection{Appendix A3: Derivative Calculations For the Radial Stretch with Spiraling Map }

The Zorich transform of $R_{s}$ is defined by \[\tilde{R_s}(x_1,...,x_n)=(u_1,...,u_n) \] with \[u_i=\begin{cases}
Mm(x_1\cos(\alpha x_n)-x_2\sin(\alpha x_n))& \text{ for }i=1\\
Mm(x_1\sin(\alpha x_n)+x_2\cos(\alpha x_n))& \text{ for }i=2\\
Mmx_i&\text{ for } 3\leq i\leq n-1\\
x_n+\ln K-\frac{1}{2}\ln\left(K^2+(1-K^2)\frac{x_1^2\sin^2 M}{x_1^2+\cdots+x_{n-1}^2}\right)& \text{ for }i=n
\end{cases}, \] where
\begin{align*}
M&=M(x_1,...,x_{n-1})=\max\{|x_1|,...,|x_{n-1}| \}\text{ and,}\\
m&=m(x_1,...,x_{n-1})=\min\left\lbrace\frac{1}{|x_1\cos(\alpha x_n)-x_2\sin(\alpha x_n)|},\frac{1}{|x_1\sin(\alpha x_n)+x_2\cos(\alpha x_n)|},\frac{1}{|x_3|},...,\frac{1}{|x_{n-1}|}  \right\rbrace.
\end{align*} 

 We will discuss bounding $\tilde{R_s}'$ and  $J_{\tilde{R_s}}$ where the derivatives exist. Also, for $x_{1},...,x_{n-1}$ not all zero, \[M(u_1,...,u_{n-1})=M(x_1,...,x_{n-1}) \] since $m$ will cancel with one of the following, $(x_1\cos(\alpha x_n)-x_2\sin(\alpha x_n)),
(x_1\sin(\alpha x_n)+x_2\cos(\alpha x_n)),x_3,...,x_{n-1}$ leaving one of the $u_i$ as $\pm M$. By definition of $m$ we have that $|u_j|\leq|u_i|$ for $1\leq j\leq n-1$. We indeed have that $\mathcal{Z}\circ \tilde{R_s}=R_s\circ \mathcal{Z}$. 
A useful calculation is that if either $x_1$ or $x_2$ are not zero, then \[\frac{\sqrt{x_1^2+x_2^2}}{\sqrt{2}}\leq\max\{|x_1\cos(\alpha x_n)-x_2\sin(\alpha x_n)|,|x_1\sin(\alpha x_n)+x_2\cos(\alpha x_n)| \}\leq\sqrt{x_1^2+x_2^2}, \]  which gives us 
\begin{equation}
\label{eq:mbound}
\frac{1}{\sqrt{x_1^2+x_2^2}}\leq\min\left\{\frac{1}{|x_1\cos(\alpha x_n)-x_2\sin(\alpha x_n)|},\frac{1}{|x_1\sin(\alpha x_n)+x_2\cos(\alpha x_n)|} \right\}\leq\frac{\sqrt{2}}{\sqrt{x_1^2+x_2^2}}.
\end{equation}

For $\tilde{R_s}$ to be quasiregular we want $\tilde{R_{s}}'$ and $\left(\tilde{R_s}'\right)^{-1}$ to be bounded. For $1\leq i\leq n-1$, let 
\[A_{i}:=\left\lbrace(x_1,...,x_{n-1},x_n)\in\left[-\frac{\pi}{2},\frac{\pi}{2}\right]^{n-1}\times\R : x_i>|x_j|\text{ for }j\neq i, 1\leq j\leq n-1 \right\rbrace.  \]  

We will break our calculations into three cases, when $(x_1,...,x_n)\in A_1$, $(x_1,...,x_n)\in A_2$, and $(x_1,...,x_n)\in A_j$ for $3\leq j\leq n-1$. 

\textbf{Case I:} Suppose that $(x_1,...,x_n)\in A_1$, so that $M=x_1$.  First note that the solution sets of the equations $x_j=x_1\cos(\alpha x_n)-x_2\sin(\alpha x_n)$, $x_j=x_1\sin(\alpha x_n)+x_2\cos(\alpha x_n)$, and $x_1\cos(\alpha x_n)-x_2\sin(\alpha x_n)=x_1\sin(\alpha x_n)+x_2\cos(\alpha x_n)$  are closed, $\sigma$-finite $(n-1)$-dimensional Hausdorff measurable sets. Note that our function is not differentiable on these sets as well. The following three sub-cases address the different regions we can be in $A_1$ which are bounded by the solution sets described, or the boundary of $A_1$. 

\textbf{Sub-case a:}
Suppose that \[m=\frac{1}{x_1\cos(\alpha x_n)-x_2\sin(\alpha x_n)}, \] if we had $-m$ the derivative calculations will just have opposite signs and the bounding would work the same. Also if $M=-x_1$, the following calculations would also just be of opposite sign and will not significantly change.  

For this case we have 
\begin{align*}
u_1&=x_1,\\
u_2&=(x_1^2\sin(\alpha x_n)+x_1x_2\cos(\alpha x_n))(x_1\cos(\alpha x_n)-x_2\sin(\alpha x_n))^{-1},\\
u_i&=x_1x_i(x_1\cos(\alpha x_n)-x_2\sin(\alpha x_n))^{-1}\text{ for }3\leq i\leq n-1\text{ and,}\\
u_n&=x_n+\ln(K)-\frac{1}{2}\ln\left(K^2+\left(1-K^2\right)\frac{x_1^2\sin^2x_1}{x_1^2+\cdots+x_{n-1}^2}\right).
\end{align*} 
We now have the derivative matrix 
\[\tilde{R_s}'=\begin{pmatrix}
1 & 0&0&0&\cdots&0&0\\
(u_2)_{x_1}& (u_2)_{x_2}&0&0&\cdots&0&(u_2)_{x_n}\\
(u_3)_{x_1}& (u_3)_{x_2}&(u_3)_{x_3}&0&\cdots&0&(u_3)_{x_n}\\
\vdots& \ddots&\cdots&\vdots&\ddots&\cdots&\vdots\\
(u_{n-1})_{x_1}& (u_{n-1})_{x_2}&0 &0&\cdots&(u_{n-1})_{x_{n-1}}&(u_{n-1})_{x_n}\\ 
(u_n)_{x_1}&(u_n)_{x_2}&(u_n)_{x_3}&(u_n)_{x_4}& \cdots&(u_{n})_{x_{n-1}}&1
\end{pmatrix},  \] where 
\[
(u_2)_{x_i}=\begin{cases}
\frac{(x_1^2-x_2^2)\sin(\alpha x_n)\cos(\alpha x_n)-2x_1x_2\sin^2(\alpha x_n)}{(x_1\cos(\alpha x_n)-x_2\sin(\alpha x_n))^2}&i=1\\
x_1^2/(x_1\cos(\alpha x_n)-x_2\sin(\alpha x_n))^2&i=2\\
0 &3\leq i\leq n-1\\
\left(\alpha x_1^3+\alpha x_1x_2^2\right)/(x_1\cos(\alpha x_n)-x_2\sin(\alpha x_n))^2&i=n
\end{cases},
\]
for $3\leq j\leq n-1$ we have 
\[(u_j)_{x_i}=\begin{cases}
-x_2x_j\sin(\alpha x_n)/(x_1\cos(\alpha x_n)-x_2\sin(\alpha x_n))^2&i=1\\
x_1x_j\sin(\alpha x_n)/(x_1\cos(\alpha x_n)-x_2\sin(\alpha x_n))^2&i=2\\
0&i\neq 1,2,j,n\\
\left(x_1^2\cos(\alpha x_n)-x_1x_2\sin(\alpha x_n)\right)/(x_1\cos(\alpha x_n)-x_2\sin(\alpha x_n))^2&i=j\\
\left(\alpha x_1^2x_j\sin(\alpha x_n)+\alpha x_1x_2x_j\cos(\alpha x_n) \right)/(x_1\cos(\alpha x_n)-x_2\sin(\alpha x_n))^2&i=n
\end{cases}, \]
and 

\[(u_n)_{x_1}=\frac{(K^2-1)\left[\left(x_1\sin^2x_1+x_1^2\sin x_1\cos x_1 \right)(x_1^2+\cdots+x_{n-1}^2)^{-1}-\frac{x_1^3\sin^2x_1}{(x_1^2+\cdots+x_{n-1}^2)^2} \right]} {K^2+(1-K^2)\frac{x_1^2\sin^2x_1}{x_1^2+\cdots+x_{n-1}^2}},
\] for $2\leq i\leq n-1$ we have
\[(u_n)_{x_i}=\frac{(1-K^2)x_ix_1^2\sin^2x_1/(x_1^2+\cdots+x_{n-1}^2)^2}{K^2+(1-K^2)\frac{x_1^2\sin^2x_1}{x_1^2+\cdots+x_{n-1}^2}}, \] and 
\[(u_n)_{x_n}=1.\]

Here we will give bounding for the partial derivatives. Using the fact that $(x_1,...,x_{n})\in A_1$ and  \eqref{eq:mbound}, and the fact that we are in sub-case a), we have 
\begin{align*}
|(u_2)_{x_1}|&=\left|\frac{(x_1^2-x_2^2)\sin(\alpha x_n)\cos(\alpha x_n)-2x_1x_2\sin^2(\alpha x_n)}{(x_1\cos(\alpha x_n)-x_2\sin(\alpha x_n))^2}\right|\\
&\leq \frac{2\left(x_1^2+x_2^2+2x_1^2\right)}{x_1^2+x_2^2}\\
&\leq 2\left(1+2\right)=6.
\end{align*}
Using similar methods we have the following bounds:
\[|(u_2)_{x_i}|\leq\begin{cases}
6& i=1\\
2&i=2\\
0& 3\leq i\leq n-1\\
8|\alpha|& i=n
\end{cases}, \]
and 
\[|(u_j)_{x_i}|\leq\begin{cases}
2&i=1\\
2& i=2\\
0& 3\leq i\leq n-1,i\neq j\\
4& i=j\\
8|\alpha|& i=n
\end{cases}. \] We need to use slightly different tactics to calculate a bound  for the partial derivative $(u_n)_{x_1}$. 
We know that \[0\leq\frac{ x_1^2\sin^2 x_1}{x_1^2+\cdots+x_{n-1}^2}\leq 1, \] so that \[1\leq K^2+(1-K^2)\frac{x_1^2\sin^2x_1}{x_1^2+\cdots+x_{n-1}^2}\leq K^2. \] We will also use the fact that $x_1\geq\sin(x_1)$ for $x_1\geq0$. From here, we have 
\begin{align*}
|(u_n)_{x_1}|&=\left|\frac{(K^2-1)\left[\left(x_1\sin^2x_1+x_1^2\sin x_1\cos x_1 \right)\left(x_1^2+\cdots+x_{n-1}^2 \right)^{-1}-x_1^3\sin^2x_1\left(x_1^2+\cdots+x_{n-1}^2\right)^{-2}\right] }{K^2+(1-K^2)\frac{x_1^2\sin^2x_1}{x_1^2+\cdots+x_{n-1}^2}}\right|\\
&\leq\left|(K^2-1)\left[\left(x_1\sin^2x_1+x_1^2\sin x_1\cos x_1 \right)\left(x_1^2+\cdots+x_{n-1}^2 \right)^{-1}-x_1^3\sin^2x_1\left(x_1^2+\cdots+x_{n-1}^2\right)^{-2}\right] \right| \\
&\leq\left|K^2-1\right|\left(1+1+|x_1|^5(x_1^2+\cdots+x_{n-1}^2)^{-2} \right)\\
&\leq|K^2-1|(2+\pi/2)\leq 4(K^2-1).
\end{align*}

We also have for $2\leq i\leq n-1$ that 
\begin{align*}
|(u_n)_{x_i}|&=\left|\frac{(1-K^2)x_1^2x_i\sin^2x_1}{\left(x_1^2+\cdots+x_{n-1}^2\right)^2\left(K^2+(1-K^2)\frac{x_1^2\sin^2x_1}{x_1^2+\cdots+x_{n-1}^2} \right)} \right|\\
&\leq(K^2-1)\frac{|x_1|^3}{(x_1^2+\cdots+x_{n-1}^2)^2}\\
&\leq (K^2-1).
\end{align*}

In conclusion, we have the bounds

\[|(u_n)_{x_i}|\leq\begin{cases}
4(K^2-1)& i=1\\
(K^2-1)& 2\leq i\leq n-1\\
1 & i=n
\end{cases}.  \] 

The above bounds are not sharp, but for our result of $\tilde{R_s}$ to be quasiregular, all we need to know is that these partial derivatives are bounded above by some constant value,  so that $\|\tilde{R_s}'\|$ is bounded. We also want $J_{\tilde{R_s}}$ to be bounded from below, so that we can use Theorem \ref{thm: removability}. To the end of bounding $J_{\tilde{R_s}}$ from below, we can notice that the only terms that appear without an $\alpha$ multiplying them occur when we multiply the diagonal of $\tilde{R_s}'$ together. That is, the non-alpha term of the Jacobian is \[Q:=\frac{x_1^2(x_1^2\cos(\alpha x_n)-x_1x_2\sin(\alpha x_n))^{n-3}}{(x_1\cos(\alpha x_n)-x_2\sin(\alpha x_n))^{2(n-2)}}. \] Notice that using \eqref{eq:mbound}, we have that
 \begin{align*}
Q&\geq\frac{x_1^{n-1}}{(x_1\cos(\alpha x_n)-x_2\sin(\alpha x_n))^{n-1}}\\
&\geq \frac{x_1^{n-1}}{(\sqrt{x_1^2+x_2^2})^{n-1}}\\
&\geq \frac{x_1^{n-1}}{(\sqrt{2x_1^2})^{n-1}}=2^{-(n-1)/2}.
\end{align*} For here, we can choose $\alpha$ so that $|\alpha|>0$ is sufficiently small so that the alpha terms have absolute value less than $\frac{1}{2}Q$. That is,
\[J_{\tilde{R_s}}>\frac{1}{2}Q\geq 2^{-(n+1)/2}. \] 

\textbf{Sub-case b:} For the case when \[m=\frac{1}{x_1\sin(\alpha x_n)+x_2\cos(\alpha x_n)} \] we have 

\begin{align*}
u_1&=(x_1^2\cos(\alpha x_n)-x_1x_2\sin(\alpha x_n))(x_1\sin(\alpha x_n)+x_2\cos(\alpha x_n))^{-1},\\
u_2&=x_1,\\
u_i&=x_1x_i(x_1\sin(\alpha x_n)+x_2\cos(\alpha x_n))^{-1}\text{ for }3\leq i\leq n-1\text{ and,}\\
u_n&=x_n+\ln(K)-\frac{1}{2}\ln\left(K^2+\left(1-K^2\right)\frac{x_1^2\sin^2x_1}{x_1^2+\cdots+x_{n-1}^2}\right).
\end{align*} 
This gives us the derivative matrix 
\[\tilde{R_s}'=\begin{pmatrix}

(u_1)_{x_1}& (u_1)_{x_2}&\cdots&(u_1)_{x_n}\\
1 & 0&\cdots&0\\
(u_3)_{x_1}&(u_3)_{x_2}&\cdots&(u_3)_{x_n}\\
\vdots& \ddots&\cdots&\vdots\\
(u_n)_{x_1}&\cdots& (u_n)_{x_{n-1}}&1
\end{pmatrix},  \] where 
\[
(u_1)_{x_i}=\begin{cases}
\frac{(x_1^2-x_2^2)\sin(\alpha x_n)\cos(\alpha x_n)+2x_1x_2\cos^2(\alpha x_n)}{(x_1\sin(\alpha x_n)+x_2\cos(\alpha x_n))^2}&i=1\\
-x_1^2/(x_1\sin(\alpha x_n)+x_2\cos(\alpha x_n))^2&i=2\\
0&3\leq i\leq n-1\\
\left(-\alpha x_1^3-\alpha x_1x_2^2\right)/(x_1\sin(\alpha x_n)+x_2\cos(\alpha x_n))^2&i=n
\end{cases},
\]
and for $3\leq j\leq n-1$ we have 
\[(u_j)_{x_i}=\begin{cases}
x_2x_j\cos(\alpha x_n)/(x_1\sin(\alpha x_n)+x_2\cos(\alpha x_n))^2&i=1\\
-x_1x_j\cos(\alpha x_n) /(x_1\sin(\alpha x_n)+x_2\cos(\alpha x_n))^2&i=2\\
0&i\neq1,2, j,n\\
\left(x_1^2\sin(\alpha x_n)+x_1x_2\cos(\alpha x_n)\right)/(x_1\sin(\alpha x_n)+x_2\cos(\alpha x_n))^2&i=j\\
\left(-\alpha x_1^2x_j\cos(\alpha x_n)+\alpha x_1x_2x_j\sin(\alpha x_n) \right)/(x_1\sin(\alpha x_n)+x_2\cos(\alpha x_n))^2& i=n
\end{cases}. \]
The partial derivatives of $u_n$ are the same as  sub-case a. Also, by looking at the similarities we can see that all of these derivatives are bounded from above, and that we can choose $\alpha$ small enough so that $J_{\tilde{R_s}}>2^{-(n+1)/2}$. 

\textbf{Sub-case c:} Let $m=x_j^{-1}$ for $3\leq j\leq n-1$, we have \begin{align*}
u_1&=(x_1^2\cos(\alpha x_n)-x_1x_2\sin(\alpha x_n))x_j^{-1},\\
u_2&=(x_1^2\sin(\alpha x_n)+x_1x_2\cos(\alpha x_n))x_j^{-1},\\
u_i&=x_1x_ix_j^{-1}\text{ for }3\leq i\leq n-1,i\neq j,\\
u_j&=x_1\text{ and,}\\
u_n&=x_n+\ln(K)-\frac{1}{2}\ln\left(K^2+\left(1-K^2\right)\frac{x_1^2\sin^2x_1}{x_1^2+\cdots+x_{n-1}^2}\right).
\end{align*} 
We have the derivative matrix 
\[\tilde{R_s}'=\begin{pmatrix}

(u_1)_{x_1}&(u_1)_{x_{2}}& \cdots&(u_1)_{x_n}\\

\vdots& \ddots&\cdots&\vdots\\
(u_n)_{x_1}&\cdots& (u_n)_{x_{n-1}}&1
\end{pmatrix},  \] where 
\[
(u_1)_{x_i}=\begin{cases}
\left(2x_1\cos(\alpha x_n)-x_2\sin(\alpha x_n)\right)x_j^{-1}&i=1\\
-x_1\sin(\alpha x_n)x_j^{-1}&i=2\\
0&3\leq i\leq n-1,i\neq j\\
\left(-x_1^2\cos(\alpha x_n)+x_1x_2\sin(\alpha x_n) \right)x_j^{-2}& i=j,\\
\left(-\alpha x_1^2\sin(\alpha x_n)-\alpha x_1x_2\cos(\alpha x_n)\right)x_j^{-1}&i=n
\end{cases},
\]

\[
(u_2)_{x_i}=\begin{cases}
\left(2x_1\sin(\alpha x_n)+x_2\cos(\alpha x_n)\right)x_j^{-1}&i=1\\
x_1\cos(\alpha x_n)x_j^{-1}&i=2\\
0&3\leq i\leq n-1,i\neq j\\
\left(-x_1^2\sin(\alpha x_n)-x_1x_2\cos(\alpha x_n) \right)x_j^{-2}& i=j,\\
\left(\alpha x_1^2\cos(\alpha x_n)-\alpha x_1x_2\sin(\alpha x_n)\right)x_j^{-1}&i=n
\end{cases},
\]
\[
(u_j)_{x_i}=\begin{cases}
1 & i=1\\
0& i\neq 1
\end{cases},
\]
for $3\leq k\leq n-1$, $k\neq j$, we have 
\[(u_k)_{x_i}=\begin{cases}
x_kx_j^{-1}&i=1\\
0&i\neq 1,k,j \\
x_1x_j^{-1}&i=k\\
-x_1x_kx_j^{-2}&i=j
\end{cases}. \]

Note that the partial derivatives of $u_n$ are the same as in the previous two cases and are bounded. Since we are assuming that $(x_1,...,x_n)\in A_1$ where the point at the origin is not included, and that $M=x_1\neq0$, this means that $x_1>|x_i|$ for all $2\leq i\leq n-1$. For $m=1/|x_j|$ for some $j$, the definition of $m$ and \eqref{eq:mbound} give 
\begin{equation}
\label{eq:mbound2}
\frac{1}{|x_j|}\leq\min\left\lbrace\frac{1}{|x_1\cos(\alpha x_n)-x_2\sin(\alpha x_n)|},\frac{1}{|x_1\sin(\alpha x_n)+x_2\cos(\alpha x_n)|} \right\rbrace\leq\frac{\sqrt{2}}{\sqrt{x_1^2+x_2^2}}. 
\end{equation}  
Using the fact that $(x_1,...,x_{n})\in A_1$ and \eqref{eq:mbound2}, for $l=1,2$ we have 
\[|(u_l)_{x_i}|\leq\begin{cases}
6& i=1\\
2&i=2\\
0& 3\leq i\leq n-1,i\neq j\\
4&i=j\\
8|\alpha|& i=n
\end{cases}, \]
and for $k\neq j$, $3\leq k\leq n-1$, we have 
\[|(u_k)_{x_i}|\leq\begin{cases}
2&i=1,j,k\\
0& i\neq1,j,k
\end{cases}. \] 

In this case, we can calculate the Jacobian by first taking the determinate across $j$th row, so that \[J_{\tilde{R_{s}}}=(-1)^{j+1}\det\begin{pmatrix}
(u_1)_{x_2}&\cdots & (u_1)_{x_n}\\
\vdots &\ddots& \vdots \\
(u_{j-1})_{x_2}&\cdots &(u_{j-1})_{x_n}\\
(u_{j+1})_{x_2}&\cdots&(u_{j+1})_{x_n}\\
\vdots&\ddots&\vdots\\
(u_n)_{x_2}&\cdots &(u_n)_{x_n}
\end{pmatrix}. \] Now take the determinate down the column where we take the partial derivative with respect to $x_n$, then the  Jacobian is
\begin{align*}
J_{\tilde{R_{s}}}&=(u_1)_{x_2}(u_2)_{x_j}\left(\prod_{\substack{3\leq i\leq n-1\\i\neq j}}(u_i)_{x_i} \right)-(u_1)_{x_j}(u_2)_{x_2}\left(\prod_{\substack{3\leq i\leq n-1\\i\neq j}}(u_i)_{x_i} \right)\\&+(-1)^{n+1}(u_1)_{x_n}\det\begin{pmatrix}
(u_2)_{x_1}&\cdots & (u_2)_{x_{n-1}}\\
\vdots&\ddots&\vdots\\
(u_n)_{x_1}& \cdots& (u_n)_{x_{n-1}}
\end{pmatrix}+(-1)^n(u_2)_{x_n}\begin{pmatrix}
(u_1)_{x_1}& \cdots &(u_1)_{x_{n-1}}\\
(u_3)_{x_1}&\cdots & (u_3)_{x_{n-1}}\\
\vdots &\ddots&\vdots \\
(u_n)_{x_1}& \cdots& (u_n)_{x_{n-1}}
\end{pmatrix}  ,
\end{align*} 
 so that the term without being multiplied by $\alpha$ will  be 
\begin{align*}
&\frac{(x_1^2\cos(\alpha x_n)-x_1x_2\sin(\alpha x_n))x_1\cos(\alpha x_n) x_1^{n-4}+(-x_1^2-x_1x_1\cos(\alpha x_n))(-x_1\sin(\alpha x_n))x_1^{n-4}}{x_j^{n-1}}\\
&=\frac{x_1^{n-1}(x_1^{3}\sin^2(\alpha x_n)+x_1^2\cos^2(\alpha x_n))}{x_j^{n-1}}=\frac{x_1^{n-1}}{x_j^{n-1}}>1.
\end{align*}

We need $\alpha$ to be sufficiently small where we have \begin{align*}
J_{\tilde{R_s}}&>\frac{1}{2}\frac{(x_1^2\cos(\alpha x_n)-x_1x_2\sin(\alpha x_n))x_1\cos(\alpha x_n) x_1^{n-4}}{x_j^{n-1}}\\
&+\frac{(-x_1^2-x_1x_1\cos(\alpha x_n))(-x_1\sin(\alpha x_n))x_1^{n-4}}{x_j^{n-1}}\\
&>\frac{1}{2}>2^{-(n+1)/2}.
\end{align*}

The last inequality shows that all we need  do is to choose $\alpha$ in finitely many cases, so that the Jacobian is bounded from below by $2^{-(n+1)/2}$. In other words, we can let $\alpha$ be the minimal in size from sub-cases a, b, and c, then we obtain $\|\tilde{R_s}'\|$ is bounded in each region. 

\textbf{Case II:} We have the case where $M=x_2$, i.e. $(x_1,...,x_n)\in A_2$, which is similar to the case when $M=x_1$. Running through similar calculations as in case I we can show that $\tilde{R_s}$ has bounded derivative matrix, where the derivative matrix is invertible. Moreove, we show that the Jacobian is bounded from below giving us that the inverse  derivative matrix is bounded as well  in the corresponding regions. 

\textbf{Case III:} Let $(x_1,..,x_n)\in A_j$ for some $3\leq j\leq n-1$, so that $M=x_j$, with $x_j>|x_i|$ for $1\leq i\leq n-1,$ $i\neq j$, and that $x_j\neq0$.  Here we will also break this case into three sub-cases  for the same reasoning as in case I. 

\textbf{Sub-case a:} Suppose that \[m=\frac{1}{x_1\cos(\alpha x_n)-x_2\sin(\alpha x_n)}, \] which means, by definition of $m$ that 
\begin{equation}
\label{eq:mbound3}
\frac{1}{x_1\cos(\alpha x_n)-x_2\sin(\alpha x_n)}\leq\frac{1}{x_j}. 
\end{equation} This means that \[x_1\cos(\alpha x_n)-x_2\sin(\alpha x_n)\geq x_j>0, \] which also implies that either $x_1\neq0$ or $x_2\neq 0$.  Since either $x_1$ or $x_2$ are not zero we have that \eqref{eq:mbound} holds. We also have the inequality 
\begin{equation}
\label{eq:Mbound}
\sqrt{x_1^2+x_2^2}\geq x_1\cos(\alpha x_n)-x_2\cos(\alpha x_n)\geq x_j. 
\end{equation}
For this case we have 
\begin{align*}
u_1&=x_j,\\
u_2&=(x_1x_j\sin(\alpha x_n)+x_2x_j\cos(\alpha x_n))(x_1\cos(\alpha x_n)-x_2\sin(\alpha x_n))^{-1},\\
u_i&=x_ix_j(x_1\cos(\alpha x_n)-x_2\sin(\alpha x_n))^{-1}\text{ for }3\leq i\leq n-1, i\neq j\\
u_j&=x_j^2(x_1\cos(\alpha x_n)-x_2\sin(\alpha x_n))^{-1}\text{ and,} \\
u_n&=x_n+\ln(K)-\frac{1}{2}\ln\left(K^2+\left(1-K^2\right)\frac{x_1^2\sin^2x_j}{x_1^2+\cdots+x_{n-1}^2}\right).
\end{align*} 
Define $\Omega:=\{1,2,j,n\}$.
We have the corresponding derivative matrix 
\[\tilde{R_s}'=\begin{pmatrix}

(u_1)_{x_1}&(u_1)_{x_{2}}& \cdots&(u_1)_{x_n}\\

\vdots& \ddots&\cdots&\vdots\\
(u_n)_{x_1}&\cdots& (u_n)_{x_{n-1}}&1
\end{pmatrix},  \] where 
\[(u_1)_{x_i}=\begin{cases}
0& i\neq j\\
1& i=j
\end{cases}, \]
\[
(u_2)_{x_i}=\begin{cases}
-x_2x_j/(x_1\cos(\alpha x_n)-x_2\sin(\alpha x_n))^2&i=1\\
x_1x_j/(x_1\cos(\alpha x_n)-x_2\sin(\alpha x_n))^2&i=2\\
0&i\notin \Omega \\
\frac{\left(x_1^2-x_2^2\right)\sin(\alpha x_n)\cos(\alpha x_n)+x_1x_2\left(\cos^2(\alpha x_n)-\sin^2(\alpha x_n)\right) }{(x_1\cos(\alpha x_n)-x_2\sin(\alpha x_n))^2} &i=j\\
\left(\alpha x_1^2x_j+\alpha x_2^2x_j\right)/(x_1\cos(\alpha x_n)-x_2\sin(\alpha x_n))^2&i=n
\end{cases},
\]

\[(u_j)_{x_i}=\begin{cases} 
-x_j^2\cos(\alpha x_n)/\left(x_1\cos(\alpha x_n)-x_2\sin(\alpha x_n)\right)^2 &i=1\\
x_j^2\sin(\alpha x_n)/\left(x_1\cos(\alpha x_n)-x_2\sin(\alpha x_n)\right)^2 & i=2\\
0 & i\notin \Omega\\
2x_j/\left(x_1\cos(\alpha x_n)-x_2\sin(\alpha x_n)\right)&i=j\\
\left(\alpha x_1x_j^2\sin(\alpha x_n)+\alpha x_2x_j^2\cos(\alpha x_n) \right)/\left(x_1\cos(\alpha x_n)-x_2\sin(\alpha x_n)\right)^2& i=n
\end{cases}, \]
for $3\leq k\leq n-1$, $k\neq j$ we have 
\[(u_k)_{x_i}=\begin{cases}
-x_kx_j\cos(\alpha x_n) /(x_1\cos(\alpha x_n)-x_2\sin(\alpha x_n))^2&i=1\\
x_kx_j\sin(\alpha x_n)/(x_1\cos(\alpha x_n)-x_2\sin(\alpha x_n))^2&i=2\\
0&i\notin \Omega\cup\{k\} \\
x_k/(x_1\cos(\alpha x_n)-x_2\sin(\alpha x_n))&i=j\\
x_j/(x_1\cos(\alpha x_n)-x_2\sin(\alpha x_n))&i=k\\
\left(\alpha x_1x_kx_j\sin(\alpha x_n)+\alpha x_2x_kx_j\cos(\alpha x_n) \right)/(x_1\cos(\alpha x_n)-x_2\sin(\alpha x_n))^2&i=n
\end{cases}, \]
and

\[(u_n)_{x_1}=\frac{(K^2-1)\left[\left(x_1\sin^2x_j\right)\left(x_1^2+\cdots+x_{n-1}^2 \right)^{-1}-x_1^3\sin^2x_j\left(x_1^2+\cdots+x_{n-1}^2\right)^{-2}\right] }{K^2+(1-K^2)\frac{x_1^2\sin^2x_j}{x_1^2+\cdots+x_{n-1}^2}},\]

\[(u_n)_{x_i}=\frac{(1-K^2)x_1^2x_i\sin^2x_j}{\left(x_1^2+\cdots+x_{n-1}^2\right)\left(K^2+(1-K^2)\frac{x_1^2\sin^2x_j}{x_1^2+\cdots+x_{n-1}^2} \right)},\]  for $2\leq i\leq n-1$, $i\neq j$, 

\[(u_n)_{x_j}=\frac{(K^2-1)\left[\left(x_1^2\sin x_j\cos x_j\right)\left(x_1^2+\cdots+x_{n-1}^2 \right)^{-1}-x_1^2x_j\sin^2x_j\left(x_1^2+\cdots+x_{n-1}^2\right)^{-2}\right] }{K^2+(1-K^2)\frac{x_1^2\sin^2x_j}{x_1^2+\cdots+x_{n-1}^2}} \] and,

\[(u_n)_{x_n}=1.\]

Note that the partial derivatives of $u_n$ are bounded using similar calculations as in case I.  Using the fact that $(x_1,...,x_n)\in A_j$ and \eqref{eq:mbound3} we have the following bounds  for $l=2,j$
\[|(u_l)_{x_i}|\leq\begin{cases}
1& i=1,2\\
0&i\neq1,2,j,n\\
6& i=j\\
4|\alpha|& i=n
\end{cases}, \]
\[|(u_1)_{x_i}|\leq \begin{cases}
1&i=1\\
0&i\neq 1
\end{cases}, \] by similar methods from Case I sub-case a) the bounds for partial derivatives of $u_n$ are
\[|(u_n)_{x_i}|\leq\begin{cases}
3|K^2-1|& i=1,j\\
2|K^2-1|&2\leq i\leq n-1, i\neq j\\
1& i=n
\end{cases}, \]
and for $k\neq 1,2,j,n$ we have 
\[|(u_k)_{x_i}|\leq\begin{cases}
1&i=1,2\\
0& i\neq1,2,j,n,k\\
4& i=j\\
2& i=k\\
4|\alpha|& i=n
\end{cases}. \]

To compute the Jacobian of $\tilde{R_{s}}$ for this case, first let \[M=\begin{pmatrix}(u_2)_{x_1}&\cdots&(u_2)_{x_{j-1}}&(u_2)_{x_{j+1}}&\cdots &(u_2)_{x_n}\\
					\vdots    &\cdots&\ddots         &\cdots         &\ddots &\vdots\\
					(u_n)_{x_n}&\cdots&(u_n)_{x_{j-1}}&(u_n)_{x_{j+1}}&\cdots &(u_n)_{x_n}  

\end{pmatrix}. \] 
Taking the determinate first row, we have \[J_{\tilde{R_{s}}}=(-1)^{j+1}\det M. \] 
Define $M_i$ to be the square matrix of order $n-2$ derived from removing the $(i-1)$th row, $2\leq i\leq n$ and $(n-1)$th column  from $M$. Taking the determinate of $M$ first along the column where the partial derivatives are taken with respect to $x_n$, we have that 
\begin{align*}
J_{\tilde{R_{s}}}&=\left(\prod_{\substack{3\leq i\leq n-1\\i\neq j}}(u_i)_{x_i} \right)\left((u_j)_{x_2}(u_2)_{x_1}-(u_j)_{x_1}(u_2)_{x_2}\right)+\sum_{i=2}^{n-1}\left[(-1)^{n+i+j+1}(u_n)_{x_i}\det(M_i)\right], 
\end{align*}
so that the  non-alpha term in $J_{\tilde{R_s}} $ is 
\begin{align*}Q&=\frac{-(-x_j^2)\cos(\alpha x_n)(x_1x_j)x_j^{n-4}(x_1\cos(\alpha x_n)-x_2\sin(\alpha x_n))^{n-4}}{\left((x_1\cos(\alpha x_n)-x_2\sin(\alpha x_n))^2\right)^{n-2}}\\&+\frac{x_j^2\sin(\alpha x_n)(-x_2x_j)x_j^{n-4}(x_1\cos(\alpha x_n)-x_2\sin(\alpha x_n))^{n-4}}{\left((x_1\cos(\alpha x_n)-x_2\sin(\alpha x_n))^2\right)^{n-2}}. \end{align*}

Using \eqref{eq:mbound3} and simplifying equations we have the following lower bound for $Q$,
\begin{align*}
Q&=\frac{x_j^{n-1}(x_1\cos(\alpha x_n)-x_2\sin(\alpha x_n))^{n-3}}{(x_1\cos(\alpha x_n)-x_2\sin(\alpha x_n))^{2n-4}}\\
&=\frac{x_j^{n-1}}{(x_1\cos(\alpha x_n)-x_2\sin(\alpha x_n))^{n-1}}\geq\frac{x_j^{n-1}}{x_j^{n-1}}=1.
\end{align*}

Then we need $\alpha$ to be sufficiently small so that 
\[J_{\tilde{R_s}}>\frac{1}{2}Q\geq\frac{1}{2}. \]
Then $J_{\tilde{R_s}}$ is bounded below and the norms of $\tilde{R_{s}}'$ and $\left(\tilde{R_s}'\right)^{-1}$ are bounded above. 

\textbf{Sub-case b: }The case when \[m=\frac{1}{x_1\sin(\alpha x_n)+x_2\cos(\alpha x_n)} \] is very similar to sub-case a. Using similar calculations we have that $J_{\tilde{R_s}}$ is bounded below, and that   $\|\tilde{R_s}'\|$  is bounded from above.\\

\textbf{Sub-case c:} Finally, we are left with our last case when we let \[m=\frac{1}{x_j}. \] We have that
\begin{align*}
u_1&=x_1\cos(\alpha x_n)-x_2\sin(\alpha x_n),\\
u_2&=x_2\sin(\alpha x_n)+x_2\sin(\alpha x_n),\\
u_i&=x_i\text{ for }3\leq i\leq n-1\text{ and,}\\
u_n&=x_n+\ln(K)-\frac{1}{2}\ln\left(K^2+\left(1-K^2\right)\frac{x_1^2\sin^2x_j}{x_1^2+\cdots+x_{n-1}^2}\right).
\end{align*}
We have the corresponding derivative matrix 
\[\tilde{R_s}'=\begin{pmatrix}

(u_1)_{x_1}&(u_1)_{x_{2}}& (u_1)_{x_3}&(u_1)_{x_4}& (u_1)_{x_5}&\cdots&(u_1)_{x_{n-1}}& (u_1)_{x_n}\\
(u_2)_{x_1}&(u_2)_{x_{2}}& (u_2)_{x_3}&(u_2)_{x_4}&(u_2)_{x_5} &\cdots&(u_2)_{x_{n-1}}&(u_2)_{x_n}\\
0        & 0         &1         & 0       & 0        & \cdots& 0           &0          \\
0        & 0         &0         & 1       & 0        & \cdots& 0           &0          \\
\vdots& \ddots       &\cdots    & \ddots  &\cdots    &\ddots &\cdots       &\vdots\\
0        & 0         &0         & 0      & 0        & \cdots& 1           &0          \\
(u_n)_{x_1}&(u_n)_{x_{2}}& (u_n)_{x_3}&(u_n)_{x_4}& (u_n)_{x_5}& \cdots& (u_n)_{x_{n-1}}&1
\end{pmatrix},  \] where 

\[(u_1)_{x_i}=\begin{cases}
\cos(\alpha x_n)& i=1\\
-\sin(\alpha x_n) & i=2\\
0& 3\leq i\leq n-1\\
-\alpha x_1\sin(\alpha x_n)-\alpha x_2\cos(\alpha x_n) & i=n
\end{cases}, \]
\[(u_2)_{x_i}=\begin{cases}
\sin(\alpha x_n)& i=1\\
\cos(\alpha x_n)& i=2\\
0&3\leq i\leq n-1\\
\alpha x_1\cos(\alpha x_n)-\alpha x_2\sin(\alpha x_n) & i=n
\end{cases}, \] and the partial derivatives for $u_n$ are the same as in the  sub-case a which we already remarked were all bounded from above. 

We have the following bounds  for the partial derivatives corresponding to $l=1,2$
\[|(u_l)_{x_i}|\leq\begin{cases}
1 & i=1,2\\
0& 3\leq i\leq n-1\\
6|\alpha|& i=n 
\end{cases}. \]

The term without $\alpha$ in $J_{\tilde{R_s}}$ is \[\cos^2(\alpha x_n)+\sin^2(\alpha x_n)=1. \] We can find an $\alpha$ sufficiently small so that \[J_{\tilde{R_{s}}}\geq\frac{1}{2}. \] Therefore, the norm of $\tilde{R_s}'$  is bounded from above in the regions where $\tilde{R_s}$ is differentiable. 

From cases I, II and III, we have that the linear distortion of $\tilde{R_s}$ is bounded from above where $\tilde{R_s}$ is differentiable.

\end{document}